\documentclass[11pt, reqno]{amsart}
\usepackage{bbm}
\usepackage[top=2.5cm,left=2.7cm,right=2.7cm,bottom=2.5cm]{geometry}
\usepackage{amsfonts}
\usepackage{amsmath,amsthm,amscd,mathrsfs,amssymb,graphicx,enumerate,
	dsfont}
\usepackage{hyperref}
\usepackage[usenames,dvipsnames]{xcolor}
\usepackage{xypic}
\usepackage{tikz-cd}
\hypersetup{colorlinks=true,citecolor=NavyBlue,linkcolor=Brown,urlcolor=Orange}

\newtheorem{thm2}{Theorem}
\newtheorem{cor2}{Corollary}
\newtheorem{conj2}{Conjecture}

\newtheorem{ques2}{Question}
\newtheorem{thm}{Theorem}[section]
\newtheorem{cor}[thm]{Corollary}
\newtheorem{lem}[thm]{Lemma}
\newtheorem{prop}[thm]{Proposition}
\newtheorem{conj}[thm]{Conjecture}
\theoremstyle{definition}
\newtheorem{defn}[thm]{Definition}

\newtheorem{ex}[thm]{Example}
\newtheorem{rmk}[thm]{Remark}

\newtheorem{ques}[thm]{Question}
\newtheorem{notation}[thm]{Notation}

\DeclareMathOperator{\Hom}{Hom}

\DeclareMathOperator{\prim}{prim}
\DeclareMathOperator{\Span}{Span}
\DeclareMathOperator{\tr}{tr}

\newcommand{\C}{\ensuremath\mathds{C}}
\newcommand{\R}{\ensuremath\mathrm{R}}
\newcommand{\N}{\ensuremath\mathds{N}}

\newcommand{\Q}{\ensuremath\mathds{Q}}

\newcommand{\FF}{\ensuremath\mathcal{F}}
\newcommand{\h}{\ensuremath\mathfrak{h}}

\newcommand{\PP}{\ensuremath\mathds{P}}

\newcommand{\HH}{\ensuremath\mathrm{H}}

\newcommand{\CH}{\ensuremath\mathrm{CH}}

\newcommand{\id}{\ensuremath\mathrm{id}}
\newcommand{\Gr}{\ensuremath\mathrm{Gr}}

\newcommand{\ZZ}{\mathbb{Z}}

\newcommand{\QQ}{\mathbb{Q}}

\newcommand{\Ss}{\mathcal S}

\newcommand{\XX}{\mathcal X}
\newcommand{\YY}{\mathcal Y}
\newcommand{\MM}{\mathcal M}

\newcommand{\VV}{\mathcal V}

\newcommand{\CC}{\mathcal C}

\newcommand{\TT}{\mathcal T}

\newcommand{\PPP}{\mathcal P}

\newcommand{\wt}{\widetilde}
\newcommand{\ima}{\hbox{Im}}

\newcommand{\rom}{\romannumeral}
\renewcommand{\1}{\mathop{\mathds{1}}\nolimits}

\newif\ifHideFoot
\HideFoottrue  % This line can be left alone.
\HideFootfalse  % Comment this line to hide footnotes

\ifHideFoot

\newcommand{\Lie}[1]{}
\newcommand{\Charles}[1]{}
\newcommand{\Robert}[1]{}

\else 

\newcommand{\marg}[1]{\normalsize{{
			\color{red}\footnote{{\color{blue}#1}}}{\marginpar[\vskip
			-.25cm{\color{red}\hfill$\Rightarrow$\tiny\thefootnote}]{\vskip
				-.2cm{\color{red}$\Leftarrow$\tiny\thefootnote}}}}}
\newcommand{\Lie}[1]{\marg{(Lie) #1}}
\newcommand{\Charles}[1]{\marg{(Charles) #1}}
\newcommand{\Robert}[1]{\marg{(Robert) #1}}

\fi

\begin{document}

	\title[MCK decompositions and varieties of K3 type]{Multiplicative
		Chow--K\"unneth decompositions \linebreak and varieties of cohomological K3
		type}

	\author{Lie Fu}
	\address{ Institut Camille Jordan, Universit\'e Claude Bernard Lyon 1, France}
	\address{IMAPP, Radboud University, Nijmegen, Netherlands}
	\email{fu@math.univ-lyon1.fr}
	
	\author{Robert Laterveer} 
	\address{CNRS - IRMA, Universit\'e de Strasbourg,
		France}
	\email{laterv@math.unistra.fr}  
	
	\author{Charles Vial}
	\address{Fakult\"at f\"ur Mathematik, Universit\"at Bielefeld, Germany} 
	\email{vial@math.uni-bielefeld.de}
	\thanks{2010 {\em Mathematics Subject Classification:} 14C15, 14C25, 14C30, 14J45, 14J29}
	
	\thanks{{\em Key words and phrases}: Algebraic cycles, Chow groups, motives, K3
		surfaces, cubic hypersurfaces, Fano varieties of lines, Franchetta conjecture,
		hyper-K\"ahler varieties, Beauville ``splitting property'' conjecture,
		multiplicative Chow--K\"unneth decomposition.}  
	
%	\date{\today}
	
	\begin{abstract} 
		Given a smooth projective variety, a Chow--K\"unneth decomposition is called
		multiplicative if it is compatible with the intersection product. 
		Following works of Beauville and Voisin, Shen and Vial conjectured that
		hyper-K\"ahler varieties admit a multiplicative Chow--K\"unneth
		decomposition. In this paper, based on the mysterious link between Fano
		varieties
		with cohomology of K3 type and hyper-K\"ahler varieties, we ask whether Fano
		varieties with cohomology of K3 type also admit a multiplicative
		Chow--K\"unneth decomposition, and provide
		evidence by establishing their existence for cubic fourfolds and K\"uchle
		fourfolds of type $c7$. 
		The main input in the cubic hypersurface case is the Franchetta property for the square 
		of the Fano variety of lines\,; this was established in our earlier work in the fourfold case and is generalized here to arbitrary dimension.
		On the other end of the spectrum, we also give
		evidence
		that varieties with ample canonical class and with cohomology of K3 type might
		admit a multiplicative Chow--K\"unneth decomposition, by establishing this for
		two families of Todorov surfaces. 
	\end{abstract}
	
	\maketitle
	
	\section{Introduction}

	\subsection{Multiplicative Chow--K\"unneth decompositions}
	Let $X$ be a smooth projective variety over a field $k$. A
	\emph{multiplicative Chow--K\"unneth decomposition} -- abbreviated MCK
	decomposition -- 
	for $X$ is a
	decomposition of the Chow motive of $X$, considered as an algebra object, that
	lifts the K\"unneth decomposition of its homological motive (when it exists).
	This notion was first introduced in \cite{SV} as a way to make explicitly
	verifiable the \emph{splitting principle} for hyper-K\"ahler varieties due to
	Beauville \cite{Beau3}. Having a multiplicative Chow--K\"unneth decomposition
	is
	a restrictive condition on the variety~$X$ and determining the class of
	varieties that could admit an MCK decomposition is still elusive. A precise
	definition of MCK decomposition is given in~\S \ref{S:MCK} and examples of
	varieties admitting or not admitting MCK decompositions are reviewed.
	Nonetheless, Beauville's splitting principle suggests that the situation for
	hyper-K\"ahler varieties is special.
	
	\begin{conj2}[Shen--Vial \cite{SV}]\label{C:MCK-HK}
		Any hyper-K\"ahler variety admits a multiplicative Chow--K\"unneth
		decomposition.
	\end{conj2}
	
	\noindent In fact, it could be moreover expected that, if it exists, such an
	MCK
	decomposition is unique for hyper-K\"ahler varieties. So far,
	Conjecture~\ref{C:MCK-HK} has been established in the hyper-K\"ahler world for
	K3
	surfaces \cite{BV} (reinterpreted by \cite[Prop.~8.14]{SV}), Hilbert schemes of
	points on K3 surfaces~\cite{V6} (see also \cite{NOY}), generalized Kummer varieties \cite{FTV} and
	all
	hyper-K\"ahler varieties birationally equivalent to these examples by
	\cite{Rie14}.
	
	\pagebreak
	
	\subsection{Fano varieties of cohomological K3 type} A smooth projective
	complex variety is said to be \emph{of cohomological K3 type}, or more
	succinctly \emph{of K3 type}, if it is even-dimensional, say $\dim X=2m$,  and
	the Hodge numbers
	$h^{p,q}(X)$ are $0$ for all $p\not=q$ except for
	$h^{m-1,m+1}(X)=h^{m+1,m-1}(X)=1$.
	Since the foundational work of Beauville--Donagi~\cite{BD}, it has become clear
	that
	hyper-K\"ahler varieties are intimately related to Fano varieties of
	cohomological K3 type. The folklore expectation seems to be that to any Fano
	variety of  K3 type can be associated geometrically (via a moduli construction)
	a hyper-K\"ahler
	variety, and that the transcendental part of the middle cohomology of the Fano
	variety corresponds to the transcendental part of the second cohomology of the
	hyper-K\"ahler variety (via an Abel--Jacobi isomorphism). Apart from the
	example
	of Beauville--Donagi \cite{BD}, this ``folklore expectation'' is based on the
	examples 
	\cite{OG06}, \cite{DV},  \cite{IM}, \cite{LLSS}, \cite{LSV},
	\cite{IM2}, \cite{IKKR}, \cite{FM}. It is further motivated by
	the construction of a closed holomorphic 2-form on the non-singular locus of
	moduli spaces associated to any Fano variety
	of cohomological K3 type~\cite{KM}.
	
	In this paper, we give further evidence that the conjectural existence of an
	MCK
	decomposition for hyper-K\"ahler varieties could transfer to Fano varieties of
	K3
	type, and thereby strengthen the apparent connection between those two types of
	varieties. This question was already tackled for certain Fano varieties of K3
	type in  \cite{d3}, \cite{Ver}, \cite{S2},
	\cite{B1B2}, where the so-called \emph{Franchetta property} (\emph{cf.}~Definition~\ref{def:Franchetta})
plays a crucial role. In our previous work~\cite{FLV}, we established the Franchetta property for the Fano variety of lines on a cubic fourfold and its square. Here we provide a generalization to smooth cubic hypersurfaces of \emph{any} dimension\,:

	\begin{thm2}[see Theorem~\ref{flv}] \label{thm2:flv}
		Let $B$
		be the open subset of $\PP\HH^0(\PP^{n+1},\mathcal{O}(3))$
		parameterizing smooth cubic hypersurfaces, and let $\FF\to B$ be the universal
		family
		of their Fano varieties of lines. Then the families 
		$\FF\to B$ and $\FF\times_B \FF\to B$
		have the \emph{Franchetta property}\,:
		for all fibers $F$ of $\FF \to B$, the images of the restriction maps 
		$$\operatorname{CH}^*(\FF) \to \operatorname{CH}^*(F) 
		\quad \text{and} \quad
		\operatorname{CH}^*(\FF\times_B\FF) \to \operatorname{CH}^*(F\times F)
	$$ 
		inject, via the cycle class map, in the cohomology rings $\HH^*(F,\QQ)$ and $\HH^*(F\times F,\QQ)$, respectively.
	\end{thm2}

As a consequence, we deduce the following, which can also be obtained from previous results of Diaz~\cite{DiazIMRN}\,:
	
	\begin{cor2}[Diaz \cite{DiazIMRN}]\label{T:cubic-$c7$} All smooth
		cubic hypersurfaces have a
		multiplicative Chow--K\"unneth decomposition.
	\end{cor2}
	
	In particular, this applies to smooth cubic fourfolds, which are the Fano fourfolds that occur in the foundational work of Beauville--Donagi~\cite{BD}.
	Corollary~\ref{T:cubic-$c7$} is stated in a more precise form in Theorem~\ref{thm:MCKcubic} and is proved in \S \ref{SS:proofMCK}, where we also explain the connection with the work of Diaz.\medskip

	Finally, we obtain  in Theorem~\ref{main3} yet another positive answer for another type of Fano varieties
	of K3 type, namely  K\"uchle fourfolds of type $c7$. The method of proof also involves the Franchetta property.

	\subsection{Varieties with ample canonical class and of cohomological K3 type}
	On the other end of the spectrum, we also provide examples of varieties of K3
	type with ample canonical class that admit an MCK decomposition\,:

	\begin{thm2} \label{T:Todorov}
		Let $S$ be a smooth
		Todorov
		surface with fundamental invariants (0,9) or (1,10). 
		Then $S$ has
		a multiplicative Chow--K\"unneth decomposition.
	\end{thm2}
	
	The proof again involves, among other things, establishing the Franchetta property for some families. Theorem~\ref{T:Todorov} gives the first example of a regular surface of general type
	with $p_g\neq 0$ that admits an MCK decomposition.
	For details on Todorov surfaces, we refer to \S \ref{sst} and to the
	references therein.
	\medskip

	We are led to ask\,: 	
	\begin{ques2}\label{K3type} Let $X$ be a smooth projective variety whose
		canonical divisor is ample or anti-ample.
		Assume that $X$ is of cohomological K3 type. Does $X$ have a multiplicative
		Chow--K\"unneth decomposition\,? If it exists, is it unique\,?
	\end{ques2}
	
	\noindent We note that without the assumption that $X$ be of cohomological K3
	type, the question has a negative answer. In the case where the canonical
	divisor is ample, a very general curve of genus larger than $2$ already
	provides
	an example that does not admit an MCK decomposition\,; see
	Example~\ref{ex:NoMCK}. Other examples are provided by very general surfaces in
	$\PP^3$ of degree $\geq 7$\,; see Proposition~\ref{P:SurfaceNoMCK}.
	In the case where the canonical divisor is anti-ample, examples are provided by
	Beauville's examples of Fano threefolds that do not satisfy the so-called
	\emph{weak splitting property}\,; see Example~\ref{ex:NoMCK-Fano}.

	\subsection{Organization of the paper} We start in \S \ref{S:MCK} by reviewing
	the notion of MCK decomposition and examples of varieties for which such an MCK
	decomposition exists. The case of curves and regular surfaces is then
	extensively reviewed in Sections \ref{S:curves} and~\ref{S:surfaces}, respectively.
 The results of
	Sections~\ref{S:MCK},~\ref{S:curves} and~\ref{S:surfaces} are mostly expository
	and serve as motivation for the special role that hyper-K\"ahler varieties and
	varieties of K3 type with ample or anti-ample canonical class play with respect
	to MCK decompositions. Our new results are contained
	in the subsequent Sections~\ref{S:cubic},~\ref{S:$c7$} and~\ref{S:todorov}
	where
	we establish the existence of an MCK decomposition for smooth cubic
	hypersurfaces,
	K\"uchle fourfolds of type $c7$, and certain Todorov surfaces, respectively.

	\subsection{Future work}
	In concomitant work, we use Corollary~\ref{T:cubic-$c7$} on the existence of an
	MCK
	decomposition for smooth cubic fourfolds to establish in \cite{FLV2} the
	generalized Franchetta conjecture for Lehn--Lehn--Sorger--van Straten
	hyper-K\"ahler eightfolds 
	and to
	study in \cite{FV} 
	the Chow motives, as algebra objects, of smooth cubic
	fourfolds with Fourier--Mukai equivalent Kuznetsov categories. Furthermore, we
	will use Theorem~\ref{thm2:flv} 
	in  \cite{FLV2} to compute the Chow motive of the
	Fano variety of lines on a smooth cubic hypersurface in terms of the Chow motive
	of the cubic hypersurface.

	\subsection{Notation and conventions} 
	In this note, the word {\sl
		variety\/} will refer to a reduced irreducible separated scheme of finite type
	over $\C$.  	We will write $\HH^j(X)$ to indicate its rational singular
	cohomology group
	$\HH^j(X(\C),\QQ)$.
	For a scheme of finite type over a field, $\CH^i(X)$
	denotes the Chow group of codimension-$i$ cycle classes on $X$ with rational
	coefficients. 
	The category of rational Chow motives (pure motives with respect to
	rational equivalence as in \cite{andre}) is denoted by $\MM_{\rm
		rat}$, which is a pseudo-abelian rigid tensor category, whose tensor unit is
	denoted by $\1$. The contravariant functor from the category of smooth
	projective varieties to $\MM_{\rm rat}$ that sends a variety to its Chow motive
	is denoted by $\h$.

	\section{Generalities on multiplicative Chow--K\"unneth decompositions}
	\label{S:MCK}
	
	\subsection{Chow--K\"unneth decomposition}
	
	\begin{defn}[Chow--K\"unneth decomposition]
		\label{def:CK}
		Let $X$ be a smooth projective variety of dimension $d$. A
		\emph{Chow--K\"unneth} decomposition for $X$ is a direct-sum decomposition
		$$\h(X)=\h^{0}(X)\oplus\cdots\oplus \h^{2d}(X)$$
		of	its rational Chow motive in $\MM_{\rm
			rat}$,
		such that for any $0\leq i\leq 2d$, the Betti realization
		$\HH^{*}(\h^{i}(X))=\HH^{i}(X)$.
		
		In other words, a Chow--K\"unneth decomposition is a system of
		self-correspondences $\left\{\pi^{0}, \dots, \pi^{2d}\right\}$ in
		$\CH^{d}(X\times X)$ satisfying the following properties\,:
		\begin{itemize}
			\item  (Projectors) $\pi^{i}\circ \pi^{i}=\pi^{i}$ for any $i$\,;
			\item (Orthogonality) $\pi^{i}\circ \pi^{j}=0$ for any $i\neq j$\,;
			\item (Completeness) $\pi^{0}+\cdots+\pi^{2d}=\Delta_{X}$\,;
			\item (K\"unneth property) $\pi^{i}_{*}\HH^{*}(X)=\HH^{i}(X)$ for any $i$. 
		\end{itemize} 
	\end{defn}
	
	The existence of a Chow--K\"unneth decomposition for any smooth projective
	variety is part of Murre's conjectures \cite{Mur}.

	\begin{rmk}[$\pi^{0}$ and $\pi^{2d}$]
		\label{rmk:pi0pi2d}
		In a Chow--K\"unneth decomposition of a $d$-dimensional irreducible smooth
		projective
		variety $X$, the first and the last projectors are usually taken to be of the
		form $\pi^{0}=z\times 1_{X}$ and $\pi^{2d}=1_{X}\times z'$ respectively, where
		$z, z'$ are 0-cycles of degree 1 and $1_{X}$ is the fundamental class. We
		point
		out that if $X$ is Kimura finite-dimensional \cite{Kimura}, we have
		$\h^{0}(X)\simeq \1$ and $\h^{2d}(X)\simeq \1(-d)$, therefore $\pi^{0}$ and
		$\pi^{2d}$ must be of the above form. 
	\end{rmk}
	
	\begin{rmk}[Duality]\label{rmk:DualCK}
		Thanks to the motivic Poincar\'e duality $\h(X) = \h(X)^{\vee}(-d)$,
		we see that a Chow--K\"unneth decomposition
		$$\h(X)=\h^{0}(X)\oplus\cdots\oplus
		\h^{2d}(X),$$ naturally admits a \emph{dual decomposition}\,:
		$$\h(X)=\h^{2d}(X)^{\vee}(-d)\oplus\cdots\oplus \h^{0}(X)^{\vee}(-d).$$
		In terms of projectors, the dual of a system $\left\{\pi^{0}, \ldots,
		\pi^{2d}\right\}$ is
		$\left\{{^{t}}\pi^{2d}, \ldots, {}^{t}\pi^{0}\right\}$.
		A Chow--K\"unneth decomposition is called \emph{self-dual} if for any $0\leq
		i\leq 2d$, we have
		$\h^{i}(X)^{\vee}=\h^{2d-i}(X)(d)$, or equivalently,
		$\pi^{i}={}^{t}\pi^{2d-i}$.
	\end{rmk}
	
	\subsection{Murre's conjectures and the Bloch--Beilinson filtration}
	
	\begin{conj}[Murre \cite{Mur}]\label{conj:Murre}
		Let $X$ be a smooth projective variety of dimension $d$. Then 
		\begin{enumerate}[(A)]
			\item there exists a Chow--K\"unneth decomposition $\left\{\pi^{0}, \ldots,
			\pi^{2d}\right\}$. 
					\end{enumerate} 
			Any such decomposition induces a descending filtration
			$$F^{j}\CH^{i}(X):=\bigcap_{k>2i-j}\ker(\pi^{k}_* : \CH^i(X) \to
			\CH^i(X))=\sum_{k\leq 2i-j}\operatorname{im}(\pi^{k}_{*}: \CH^i(X) \to
			\CH^i(X))$$
			with the following properties\,:
			\begin{enumerate}[(A)]
			\item[(B)] $F^{0}\CH^{i}(X)=\CH^{i}(X)$ and $(B')$ $F^{i+1}\CH^i(X) =  0$.
			\item[(C)] The filtration $F^{\bullet}$ on $\CH^{i}(X)$ is independent of the
			choice
			of the Chow--K\"unneth decomposition. 
			\item[(D)] $F^{1}\CH^{i}(X)=\CH^{i}(X)_{\hom}$.
		\end{enumerate} 
	\end{conj}
	
	We note that Murre's conjecture for all smooth projective varieties is
	equivalent to the existence of the Bloch--Beilinson filtration\,; see \cite{J}
	for
	a precise statement. In particular, any Chow--K\"unneth decomposition induces a
	splitting of the conjectural Bloch--Beilinson filtration. As will be explained
	below in Remark~\ref{rmk:bigrading}, the following notion of multiplicative
	Chow--K\"unneth decomposition gives a sufficient condition for the above
	splitting to be compatible with intersection product.

	\subsection{Multiplicative Chow--K\"unneth decomposition} 
	
	Recall that if $X$ is a $d$-dimensional irreducible smooth scheme of finite
	type over a field, intersection
	product defines a (graded) ring structure on $\CH^\ast(X)=\bigoplus_i
	\CH^i(X)$\,;
	moreover, if $X$ is in addition proper, the intersection product is controlled
	by the class of the small diagonal $\delta_X := \{(x,x,x) \in X^3\}$ in
	$\CH^{2d}(X\times X\times X)$ in the sense that $(\delta_X)_*(\alpha\times
	\beta) = \alpha \cdot \beta$ for all $\alpha$ and $\beta \in \CH^*(X)$, where $\delta_X$ is viewed as a correspondence from $X\times X$ to $X$, or equivalently as a morphism $\h(X\times X)\to \h(X)$. Together with the canonical isomorphism $\h(X\times X)\simeq \h(X)\otimes \h(X)$, the small diagonal $\delta_X$
	endows the Chow motive $\h(X)$ with the structure of a unital commutative
	algebra object. (The unit is the fundamental class of $X$, seen as a morphism
	$\mathds{1} \to \h(X)$.) We write $$\mu : \h(X) \otimes \h(X) \to \h(X)$$ for
	the multiplication thus defined.

	\begin{defn}[Multiplicative Chow--K\"unneth (MCK) decomposition, Shen--Vial
		\cite{SV}]
		\label{def:MCK}
		Let $X$ be a smooth projective variety of dimension $d$.
		A Chow--K\"unneth decomposition 
		$$\h(X)=\h^{0}(X)\oplus \cdots \oplus \h^{2d}(X)$$ is called
		\emph{multiplicative}, if for any $0\leq i, j\leq 2d$, the restriction of the
		multiplication $\mu: \h(X)\otimes \h(X)\to \h(X)$ to the direct summand
		$\h^{i}(X)\otimes \h^{j}(X)$ factors through the direct summand
		$\h^{i+j}(X)$. 
	\end{defn}
	
	Note that a Chow--K\"unneth decomposition is always multiplicative modulo
	homological equivalence\,; the key point is to require this property modulo
	rational equivalence.
	In practice, it is useful to express the above notion in terms of projectors
	and
	correspondences.
	\begin{lem} 
		\label{lemma:MCKprojector}
		Let $X$ be a smooth projective variety of dimension $d$. Let $\{\pi^{0},
		\ldots,
		\pi^{2d}\}$ be the system of projectors corresponding to a Chow--K\"unneth
		decomposition of $X$ (see Definition~\ref{def:CK}). Then the following
		conditions are equivalent\,:
		\begin{enumerate}[$(i)$]
			\item The Chow--K\"unneth decomposition is multiplicative\,;
			\item For any $i, j, k$ such that $i+j\neq k$, we have $\pi^{k}\circ
			\delta_{X}\circ(\pi^{i}\otimes \pi^{j})=0$.
			\item For any $i, j$, we have $\pi^{i+j}\circ \delta_{X}\circ(\pi^{i}\otimes
			\pi^{j})=\delta_{X}\circ(\pi^{i}\otimes \pi^{j})$.
			\item $\delta_{X}=\sum_{i, j}\pi^{i+j}\circ \delta_{X}\circ(\pi^{i}\otimes
			\pi^{j})$.
		\end{enumerate} 
		Here $\delta_{X}$ denotes the small diagonal of $X^{3}$, viewed as a
		correspondence from $X\times X$ to $X$.
	\end{lem}
	\begin{proof}
		Noting that $\mu$ is induced by $\delta_{X}$ by definition, the equivalence
		between $(i)$ and $(ii)$ becomes tautological.\\
		$(ii)\Longrightarrow (iv)$\,: by the completeness of the system
		$\sum_{i}\pi^{i}=\Delta_{X}$, we see that 
		$$\delta_{X}=\sum_{i,j,k}\pi^{k}\circ
		\delta_{X}\circ(\pi^{i}\otimes \pi^{j})=\sum_{k=i+j}\pi^{k}\circ
		\delta_{X}\circ(\pi^{i}\otimes \pi^{j})+\sum_{k\neq i+j}\pi^{k}\circ
		\delta_{X}\circ(\pi^{i}\otimes \pi^{j})=\sum_{i,j}\pi^{i+j}\circ
		\delta_{X}\circ(\pi^{i}\otimes \pi^{j}).$$
		$(iv)\Longrightarrow (iii)$\,: it is enough to post-compose both sides of
		$(iv)$
		with $\pi^{i}\otimes \pi^{j}$ and use the orthogonality between the
		projectors.\\
		$(iii)\Longrightarrow (ii)$\,: it suffices to pre-compose both sides of
		$(iii)$
		with $\pi^{k}$ and use the orthogonality between the projectors.
	\end{proof}

	It turns out that an MCK decomposition is automatically self-dual (Remark
	\ref{rmk:DualCK})\,:
	\begin{prop}[Multiplicativity implies self-duality {\cite[\S 6 Footnote
			24]{FuVialJAG}}]
		\label{prop:Duality}
		Let $\{\pi^{0}, \ldots, \pi^{2d}\}$ be a multiplicative Chow--K\"unneth
		decomposition  for a smooth projective variety  $X$ of dimension~$d$. Then it
		is
		self-dual, that is, $\pi^{i}={}^{t}\pi^{2d-i}$ for all $i$.
	\end{prop}
	\begin{proof}
		Projecting both sides of $(iv)$ in Lemma \ref{lemma:MCKprojector} to the first
		two factors (or equivalently, pre-composing with the canonical morphism
		$\epsilon:
		\h(X)\to \1(-d)$ given by the fundamental class), one finds 
		$$\epsilon \circ \delta_{X}=\sum_{i,j}\epsilon \circ \pi^{i+j}\circ
		\delta_{X}\circ (\pi^{i}\otimes \pi^{j}).$$
As $\epsilon \circ \pi^k : \h(X) \to \1(-d)$ is a rational multiple of the fundamental class for all $k$, we get that $\epsilon \circ \pi^{i+j}=0$ for all $i+j\neq 2d$,  and $\epsilon\circ \pi^{2d} =\epsilon$.
		Therefore the equality simplifies to 
		$$\epsilon \circ \delta_{X}=\sum_{i}\epsilon \circ \delta_{X}\circ
		(\pi^{2d-i}\otimes \pi^{i}).$$
		Now noting that $\epsilon \circ \delta_{X}$ is the diagonal class
		$\Delta_{X}\in \CH^{d}(X\times X)$, we obtain
		$$\Delta_{X}=\sum_{i} ({}^{t}\pi^{2d-i}\otimes\pi^{i})_{*}\Delta_{X}.$$
		In other words, $\id=\sum_{i}{}^{t}\pi^{2d-i}\circ \pi^{i}$. This allows us to
		conclude by composing with $\pi^{i}$ and ${}^{t}\pi^{2d-i}$\,:
		$$\pi^{i}=\left(\sum_{j}{}^{t}\pi^{2d-j}\circ \pi^{j}\right)\circ
		\pi^{i}={}^{t}\pi^{2d-i}\circ\pi^{i}=
		{}^{t}\pi^{2d-i}\circ\left(\sum_{j}{}^{t}\pi^{2d-j}\circ
		\pi^{j}\right)={}^{t}\pi^{2d-i}.$$
	\end{proof}
	
	\begin{rmk}[Multiplicative bigrading]
		\label{rmk:bigrading}
		A multiplicative Chow--K\"unneth decomposition $\h(X)=\h^{0}(X)\oplus
		\cdots\oplus \h^{2d}(X)$ naturally gives rise to a multiplicative bigrading on
		the Chow ring $\CH^{*}(X)=\bigoplus_{i, s} \CH^{i}(X)_{(s)}$ with
		$$\CH^{i}(X)_{(s)}:=\CH^{i}\left(\h^{2i-s}(X)\right):=\Hom\left(\1(-i),
		\h^{2i-s}(X)\right).$$
		Here the multiplicativity means that
		\begin{equation}\label{eq:weaksplitting}
		\CH^{i}(X)_{(s)}\cdot \CH^{i'}(X)_{(s')}\subseteq \CH^{i+i'}(X)_{(s+s')},
		\end{equation}
		which clearly follows from the multiplicativity of the Chow--K\"unneth
		decomposition. 
		However, we note that actually any Chow--K\"unneth decomposition induces a
		splitting of the Chow groups of $X$ and the property that it is a ring
		grading,
		\emph{i.e.}, that $\CH^{i}(X)_{(s)}\cdot \CH^{i'}(X)_{(s')}\subseteq
		\CH^{i+i'}(X)_{(s+s')}$ is 
		strictly weaker than having a MCK decomposition\,; a very general curve of
		genus $\geq 2$ does
		not admit an MCK decomposition (Example~\ref{ex:NoMCK} below) but any
		splitting
		of its Chow groups induced by a Chow--K\"unneth decomposition is compatible
		with the
		intersection product.
		The new grading is chosen so that, via Murre's conjecture
		\ref{conj:Murre} (C), $$\CH^{i}(X)_{(s)}=\Gr_{F}^{s}\CH^{i}(X).$$
		In other words, a multiplicative Chow--K\"unneth decomposition should induce a
		multiplicative \emph{splitting} of the (conjectural) Bloch--Beilinson
		filtration on Chow
		groups.

		We call the indexation by $s$ the \emph{grade} of a cycle. For example, by
		Murre's conjecture \ref{conj:Murre} (B) and (D), all cycles of negative grade
		are expected to be zero and the subspace consisting of cycles of grade zero,
		$\CH^{*}(X)_{(0)}$, is expected to inject into the cohomology of $X$ via the
		cycle class map.  In particular, the subalgebra of $\CH^*(X)$ generated by
		$\CH^1(X)_{(0)}$ is expected to inject into cohomology\,; this is Beauville's
		so-called \emph{weak splitting property}~\cite{Beau3}.
	\end{rmk}
	
	\subsection{Which varieties admit a multiplicative Chow--K\"unneth
		decomposition\,?}
	Although a Chow--K\"unneth decomposition is conjectured to exist for all smooth
	projective varieties, there exist examples of varieties (in fact, examples of
	curves\,; see \emph{e.g.}\ Example~\ref{ex:NoMCK} below) that do not admit any
	MCK decomposition.
	Nonetheless, as shown by Shen--Vial~\cite{SV2}, the
	notion of multiplicative Chow--K\"unneth decomposition is robust enough to
	allow
	many standard procedures to produce new examples out of old ones.
	\begin{prop}[\cite{SV2}]
		\label{prop:SVmachine}
		Let $X$ and $Y$ be smooth projective varieties admitting  MCK decompositions
		$\{\pi^i_X\}$ and $\{\pi^i_Y\}$, respectively.
		\begin{itemize}
			\item  (Product) The product $X\times Y$ has a naturally induced MCK
			decomposition\,: for all $k$, $\pi_{X\times
				Y}^{k}:=\sum_{i=0}^{k}\pi^{i}_{X}\otimes \pi^{k-i}_{Y}$.
			\item (Projective bundle) If $E$ is a vector bundle on $X$ whose Chern
			classes
			are all of grade~0.
			then $\PP(E)$ has a natural MCK decomposition.
			\item (Blow-up) Suppose that $Y$ is a subvariety of $X$ and that, as an
			element in
			$\CH(X\times Y)$, the graph of the embedding is of grade~0 for the natural
			product MCK decomposition. Assume further that the Chern
			classes of the normal bundle are of grade 0 and the Chern classes of the
			tangent
			bundle of $X$ are of grade 0. Then the blow-up of $X$ along $Y$ admits a
			natural
			MCK decomposition.
			\item (Quotient) If a finite group $G$ acts on $X$ such that the graphs of
			the
			automorphisms (as elements in $\CH(X\times X)$) are of grade 0, then the
			quotient $X/G$ admits a natural MCK decomposition.
			\item (Hilbert scheme) Assume that the Chern classes of $X$ are of grade~0.
			Then the Hilbert schemes of length-2 and length-3 subschemes of $X$ admit an
			MCK.
		\end{itemize}
	\end{prop}
	
	Consequently, we may ask the following general, but vague, question\,:
	
	\begin{ques}
		\label{quest:MCK}
		Which smooth projective varieties constitute the ``building blocks'' of
		varieties admitting an MCK decomposition\,? For which of those can we expect an MCK decomposition to be
		unique\,?
	\end{ques}

	Note that, since homological equivalence and rational equivalence agree on
	powers of a
	variety whose motive is of Lefschetz type (for example, toric varieties,
	homogeneous varieties), such a variety admits a unique Chow--K\"unneth decomposition and
	this decomposition is multiplicative. \medskip
	
	The original and main motivation for studying MCK decompositions comes from the
	study of Chow rings of
	varieties with trivial canonical bundle. 
	A canonical multiplicative
	Chow--K\"unneth decomposition exists in the following cases\,: abelian
	varieties
	\cite{Beau}, \cite{DM}, \cite{MR1265530}, K3 surfaces \cite{BV} (interpreted by
	\cite[Proposition 8.4]{SV}) and some (conjecturally, \emph{all}) hyper-K\"ahler
	varieties \cite{SV}, \cite{V6}, \cite{FTV}, \cite{FLV}. 
	The case of Calabi--Yau varieties is not so clear\,: there are examples of
	Calabi--Yau
	varieties due to Beauville~\cite[Example
	2.1.5(b)]{Beau3} that do not 
	admit a
	Chow--K\"unneth decomposition inducing a grading
	satisfying~\eqref{eq:weaksplitting}, while examples of Calabi--Yau varieties
	with a MCK decomposition exist in all dimensions~\cite{LV}.
	Concerning the uniqueness of an MCK decomposition, we will show that this is
	the
	case for curves and for regular surfaces with finite-dimensional motive in the
	sense of Kimura. It is
	expected \cite{SV, FuVialJAG} that for hyper-K\"ahler varieties, if an MCK
	decomposition exists, then it is unique. Note however that an MCK decomposition
	is \emph{not} unique for abelian varieties\,; translating the canonical
	Chow--K\"unneth decomposition (which is multiplicative) of an abelian variety
	along a point that is not rationally equivalent to the origin provides a new
	MCK
	decomposition.\medskip

	Concerning the existence of an MCK decomposition in general, the answer is for
	the time
	being not clear in general.
	The main purpose of the paper is, beyond reviewing known
	examples of varieties admitting or not admitting an MCK decomposition, to
	explore whether varieties with ample or anti-ample canonical class 
	can be added to the list of varieties admitting such a
	decomposition.
	
	In the case of Fano varieties, there are examples of varieties (with motive not
	of Lefschetz type) that admit an MCK decomposition. By \cite[Proposition
	5.7]{FuVialJAG}, all Fermat cubic hypersurfaces admit an MCK
	decomposition. It is in fact conjectured  in
	\cite[Conjecture 5.8]{FuVialJAG} that all Fano or Calabi--Yau Fermat
	hypersurfaces admit MCK decompositions. Other
	examples have been exhibited in~\cite{d3}, \cite{Ver}, \cite{S2},
	\cite{B1B2} and Corollary~\ref{T:cubic-$c7$} shows that cubic fourfolds can be
	added to the list.
	On the other hand there are also examples of Fano varieties that do \emph{not}
	admit an MCK decomposition\,:

	\begin{ex}[Fano varieties may fail to have an MCK decomposition]
		\label{ex:NoMCK-Fano}
		In \cite[Example~2.1.5(a)]{Beau3}, Beauville constructed, by blowing up $\PP^3$
		along certain smooth curves, examples of Fano 3-folds such that the subalgebra
		generated by divisors does not inject in cohomology. Such Fano 3-folds cannot
		have an MCK decomposition. Indeed, by \cite[\S 4.2.2]{V4}, any Chow--K\"unneth
		decomposition of the blow-up $X$ of a curve inside $\PP^3$ will be such that
		$\CH^1(X) = \CH^1(X)_{(0)}$ and such that $\CH^*(X)_{(0)} \hookrightarrow
		\HH^*(X)$ is injective. In particular, for any choice of Chow--K\"unneth
		decomposition, $X$ fails to satisfy Beauville's weak splitting property, and
		hence by Remark~\ref{rmk:bigrading} any choice of Chow--K\"unneth decomposition
		for $X$ fails to be multiplicative.
	\end{ex}

	In the case of varieties with ample canonical bundle, we will review below that
	a very general curve of genus $>2$ (Example~\ref{ex:NoMCK})  does \emph{not}
	admit an MCK decomposition.
	Moreover, Proposition~\ref{P:SurfaceNoMCK} below suggests that a general
	surface in $\PP^3$ of degree $\geq 7$ also does not admit an MCK decomposition. 
	There are however examples of varieties with ample canonical bundle (and with
	motive not of Lefschetz type) that admit an MCK decomposition\,: this is the
	case for example for any product of hyperelliptic curves of genus $>1$ by
	combining Example~\ref{E:GS} with Proposition~\ref{prop:SVmachine}. 
	Our Theorem~\ref{T:Todorov} provides further examples, namely Todorov surfaces
	of type $(0,9)$ or $(1,10)$, that are not (birational to) products of curves.
	\medskip

	It turns out that the Todorov surfaces we study are intimately linked to K3
	surfaces. Likewise the Fano examples of \cite{d3}, \cite{Ver}, \cite{S2},
	\cite{B1B2}, as well as the cubic fourfolds of Corollary~\ref{T:cubic-$c7$}, have
	\emph{cohomology of K3 type}\,: these are Fano fourfolds with 
	Hodge numbers
	$h^{p,q} =0$ for all $p\not=q$ except for
	$h^{3,1}(X)=h^{1,3}(X)=1$. In the light of these examples, but also based (in
	the Fano case) on the folklore expectation that to any Fano
	variety of  K3 type can be associated geometrically (via a moduli construction)
	a hyper-K\"ahler
	variety, we ask whether every smooth projective variety with ample or anti-ample
	canonical bundle, and with cohomology of K3 type, admits a (unique)
	multiplicative Chow--K\"unneth decomposition\,; \emph{cf.}\
	Question~\ref{K3type}.

	\section{Curves}\label{S:curves}
	
	\subsection{Multiplicative Chow--K\"unneth decomposition for curves}
	As mentioned before, the projective line (a special case of homogeneous
	varieties) and elliptic curves (special cases of abelian varieties) admit
	canonical multiplicative Chow--K\"unneth decompositions. 
	Let $C$ be a smooth projective curve of genus $g\geq 2$. As it is Kimura
	finite-dimensional, by Remark \ref{rmk:pi0pi2d} and Proposition
	\ref{prop:Duality}, a multiplicative Chow--K\"unneth decomposition for $C$ must
	take the following form\,:
	\begin{equation}\label{eqn:CKCurve}
	\pi^{0}=z\times 1_{C}, \pi^{1}=\Delta_{C}-z\times 1_{C}-1_{C}\times z,
	\pi^{2}=1_{C}\times z,
	\end{equation}
	where $z$ is a 0-cycle of degree 1 and $1_{C}$ is the fundamental class. Given
	such a 0-cycle $z$, there is the natural embedding $\iota: C\to J(C)$, which
	sends a point $p\in C$ to $\mathcal{O}_{C}(p-z)$. Denote
	$[C]:=\iota_{*}(1_{C})\in \CH_{1}(J(C))$. 
	
	\begin{prop}[MCK decomposition for curves]
		\label{prop:MCKCurve}
		Notation is as above ($g\geq 2$).  Let $z$ be a 0-cycle of degree~1 on $C$.
		Then
		the following conditions are equivalent\,:
		\begin{enumerate}[$(i)$]
			\item the Chow--K\"unneth decomposition \eqref{eqn:CKCurve} determined by $z$
			is
			multiplicative\,;
			\item the modified small diagonal $\Gamma_{3}(C, z)=0$ in $\CH^{2}(C^{3})$,
			where $$\Gamma_{3}(C,
			z):=\delta_{C}-p_{12}^{*}(\Delta_{C})p_{3}^{*}(z)-p_{23}^{*}(\Delta_{C})p_{1}^{*}(z)-p_{13}^{*}(\Delta_{C})p_{2}^{*}(z)+p_{1}^{*}(z)p_{2}^{*}(z)+p_{1}^{*}(z)p_{3}^{*}(z)+p_{2}^{*}(z)p_{3}^{*}(z)\,;$$
			\item the class $[C]$ belongs to $\CH_{1}(J(C))_{(0)}$.
		\end{enumerate}
		In particular, if it exists, an MCK decomposition for $C$ is unique, given by
		$z=\frac{1}{2g-2}K_{C}$.
	\end{prop}
	\begin{proof}
		The equivalence between $(i)$ and $(ii)$ follows from a direct computation
		using Lemma \ref{lemma:MCKprojector} (see
		\cite[Proposition 8.14]{SV}).\\
		$(ii)\Longrightarrow (iii)$ is proved in \cite[Proposition 7.1]{FuVialJAG}
		using
		an idea from \cite[Proposition~3.2]{BV}. Let $f : C^3 \to J(C)$ be the
		composition of the embedding $\iota^{3}: C^{3} \to J(C)^{3}$ followed by the
		summation on $J(C)$.  We have $$f_*(\Gamma_{3}(C, z)) = [3]_*[C] - 3[2]_*[C] +
		3[C] = 0 \quad \text{in}\
		\CH_1(J(C)).$$ Using the Beauville decomposition \cite{Beau} of $\CH_{1}(J(C))$,
		we see that $[C]$ belongs to $\CH_1(J(C))_{(0)}$.\\
		$(iii)\Longrightarrow (i)$ is implied by \cite[Propositions 5.3 and
		6.1]{FuVialJAG}.\\
		Finally, for the uniqueness, letting $(ii)$ act on $\Delta_{C}$, we get
		$c_{1}(T_{C})=(2-2g)z$. Hence $z$ is determined by the curve.
	\end{proof}
	
	The next two examples illustrate that Proposition \ref{prop:MCKCurve} leads to
	both existence and non-existence results on MCK decompositions for curves\,:
	\begin{ex}[Curves with MCK decompositions {\cite[Example~8.16]{SV}}]
		\label{E:GS}
		If $C$ is hyperelliptic, take $z$ to be a Weierstrass point, then
		Gross--Schoen
		\cite{GrossSchoen} proved that the modified small diagonal $\Gamma_{3}(C,z)$
		vanishes. Alternately, by \cite[Proposition	2.1]{Tavakol}, the class $[C]$ 
		belongs to $\CH_{1}(J(C))_{(0)}$. By
		Proposition \ref{prop:MCKCurve}, this implies that the Chow--K\"unneth
		decomposition
		\eqref{eqn:CKCurve} is multiplicative.
	\end{ex}
	
	\begin{ex}[Curves without MCK decompositions {\cite[\S 7]{FuVialJAG}}]
		\label{ex:NoMCK}
		As is pointed out before, any MCK decomposition for a curve $C$ is determined
		by
		a 0-cycle of degree 1. If $C$ admits an MCK decomposition, then Proposition
		\ref{prop:MCKCurve} $(iii)$ implies in particular that the \emph{Ceresa cycle}
		$[C]-[-1]_{*}[C]\in \CH_{1}(J(C))$ vanishes.
		
		Ceresa \cite{Ceresa} proved that the Ceresa cycle of a very general complex
		curve of genus $>2$ is not algebraically trivial\,; as such a very general complex curve of genus $>2$ does \emph{not} admit an MCK decomposition. 
		
		 As more explicit examples, 
		Otsubo \cite{Otsubo} proves that the Ceresa cycle of Fermat curves of degree
		$4\leq d \leq 1000$ is not algebraically trivial. Therefore, those curves do
		\emph{not}
		admit any MCK decomposition. A specialization argument (\emph{cf.}\ \cite[Lemma 3.2]{Vo}) then establishes that a very general plane curve of degree $4\leq d \leq 1000$ 
		does \emph{not} admit an MCK decomposition. 
	\end{ex}
	
	\begin{rmk}[MCK decomposition modulo algebraic equivalence]
		If instead of rational equivalence, we work with algebraic
		equivalence, the analogue of Proposition \ref{prop:MCKCurve} still holds
		and the choice of the 0-cycle $z$ becomes irrelevant. More precisely, given a
		smooth projective curve $C$, the following conditions are equivalent\,:
		\begin{enumerate}[$(i)$]
			\item $C$ admits a multiplicative Chow--K\"unneth decomposition modulo
			algebraic
			equivalence (which, again, must be of the form \eqref{eqn:CKCurve})
			\item The modified small diagonal $\Gamma_{3}(C, \operatorname{pt})$ is
			algebraically trivial.
			\item The class $[C]$ belongs to $\CH_{1}(J(C))_{(0)}/\operatorname{alg}$.
		\end{enumerate}
		Of course, $(iii)$ implies as before\\
		$(iv)$ the Ceresa cycle $[C]-[-1]_{*}[C]$ is algebraically trivial.\\
		Now the point of this remark is that $(iv)$ is actually equivalent to $(iii)$,
		hence also to $(i)$, or $(ii)$. Indeed, in the Beauville decomposition  
		$$\CH_{1}(J(C))=\bigoplus_{s=0}^{g-1}\CH_{1}(J(C))_{(s)},$$ denote by
		$C_{(s)}$
		the grade-$s$ component of the class $[C]$. It is well-defined modulo 
		algebraic equivalence. Then $(iv)$ implies that $C_{(1)}=0$. By Marini's
		result
		\cite[Corollary 26]{MR2414143}, we have $C_{(s)}$ is algebraically trivial for all $s>0$, that is,
		$[C]\in\CH_{1}(J(C))_{(0)}/\operatorname{alg}$.
		
		In conclusion, the vanishing of the Ceresa cycle characterizes the
		multiplicativity of the Chow--K\"unneth decomposition modulo algebraic
		equivalence.
	\end{rmk}

	\subsection{On the tautological ring of powers of curves}	
	The following proposition, essentially due to Tavakol~\cite{Tavakol}, relates
	for a given curve the existence of an MCK decomposition to the existence of
	enough relations in the tautological ring.
	
	\begin{prop}\label{P:inj-curves}
		Let $C$ be a smooth projective complex curve.
		Let $\operatorname{R}^\ast(C^m)\subseteq \CH^\ast(C^m)$ denote the subring
		generated by pullbacks of the canonical divisor and of the diagonal
		$\Delta_C$.
		Then $C$ admits an MCK decomposition if and only if the
		cycle class map induces injections
		\[  \operatorname{R}^{i}(C^m)\hookrightarrow\ \HH^{2i}(C^m)\]
		for all  positive integers $m$ and all $i$. 
	\end{prop}
	\begin{proof}
		Tavakol \cite{Tavakol} shows that the cohomological relations among $p_i^*K_C$
		and $p_{ij}^*\Delta_C$ are generated by 3 relations, namely the 
		Faber--Pandharipande relation,
		the Gross--Schoen relation and a relation implied by the finite-dimensionality relation\,; see also
		\cite[Remark~3.8(i)]{Yin} and \cite[\S 2.3]{FLV2}. By Kimura \cite{Kimura}, the finite-dimensionality
		relation holds modulo rational equivalence. The Gross--Schoen relation is
		nothing but the vanishing of the modified small diagonal. By
		Proposition~\ref{prop:MCKCurve}, this relation is equivalent to the existence of
		an MCK decomposition. Finally, the Faber--Pandharipande relation, $p_1^*K_C\cdot
		p_2^*K_C = \deg(K_C)\, p_1^*K_C \cdot \Delta_C$, can be obtained by making both
		sides of the Gross--Schoen relation, viewed as correspondences from $C$ to
		$C\times C$, act on $K_C$.
	\end{proof}

	\section{Regular surfaces}\label{S:surfaces}
	
	\subsection{Multiplicative Chow--K\"unneth decomposition for regular surfaces}
	Let $S$ be a \emph{regular surface}, that is, a smooth projective complex
	surface with $\HH^{1}(S, \mathcal{O}_{S})=0$. Then for any 0-cycle of degree 1
	on $S$, we have a self-dual Chow--K\"unneth decomposition 
	\begin{equation}\label{eqn:CKsurface}
	\pi^{0}=z\times 1_{S}\ ,\ \pi^{4}=1_{S}\times z\ ,\ 
	\pi^{2}=\Delta_{S}-\pi^{0}-\pi^{4}\ , \ \pi^{1}=\pi^{3}=0.
	\end{equation}
	Assuming the Kimura finite-dimensionality conjecture \cite{Kimura}, any
	self-dual Chow--K\"unneth decomposition  should be of this form (Remark
	\ref{rmk:pi0pi2d}). Similarly to
	Proposition \ref{prop:MCKCurve}, we have the following result.
	
	\begin{prop}[MCK decomposition for regular surfaces]
		\label{prop:MCKsurface}
		Let $S$ be a regular smooth projective surface. Let $z$ be a 0-cycle of degree
		1
		on $S$. Then the following conditions are equivalent\,:
		\begin{enumerate}[(i)]
			\item[(i)] the Chow--K\"unneth decomposition \eqref{eqn:CKsurface} is
			multiplicative\,;
			\item[(ii)] the modified small diagonal $\Gamma_{3}(S, z) \in \CH^{4}(S^{3})$
			vanishes, where $$\Gamma_{3}(S,
			z):=\delta_{S}-p_{12}^{*}(\Delta_{S})p_{3}^{*}(z)-p_{23}^{*}(\Delta_{S})p_{1}^{*}(z)-p_{13}^{*}(\Delta_{S})p_{2}^{*}(z)+p_{1}^{*}(z)p_{2}^{*}(z)+p_{1}^{*}(z)p_{3}^{*}(z)+p_{2}^{*}(z)p_{3}^{*}(z).$$
		\end{enumerate}
		Moreover, they imply the following two properties\,:
		\begin{enumerate}[(i)]
			\item[(iii)] $\operatorname{Im}\left(\CH^{1}(S)\otimes
			\CH^{1}(S)\xrightarrow{\cdot} \CH^{2}(S)\right)=\Q\cdot z$.
			\item[(iv)] $c_{2}(T_{S})=\chi_{\operatorname{top}}(S)\, z$, where
			$\chi_{\operatorname{top}}$ is the topological Euler characteristic.
		\end{enumerate}
		In particular, if it exists, an MCK decomposition of the form
		\eqref{eqn:CKsurface} for $S$ is unique. 
	\end{prop}
	\begin{proof}
		The equivalence between $(i)$ and $(ii)$ is a direct computation using Lemma
		\ref{lemma:MCKprojector}, which was
		first observed in \cite[Proposition~8.4]{SV}. The implication from them to
		$(iii)$ and
		$(iv)$ are proved as in \cite{BV}\,: let both sides of $(ii)$ act on the
		exterior
		product of two divisors to  obtain $(iii)$ and  on~$\Delta_{S}$ to get $(iv)$.
	\end{proof}
	
	\begin{ex}[Regular surfaces with an MCK decomposition] \label{E:surfaces}
		First we note that any Chow--K\"unneth decomposition of a surface whose Chow
		motive is of
		Lefschetz type is multiplicative. Hence any complex surface with trivial Chow
		group of 0-cycles admits an MCK decomposition. Beyond the above obvious
		examples, the following regular surfaces are known to admit an MCK
		decomposition\,:	K3 surfaces,
		certain elliptic surfaces constructed by Schreieder \cite{Schreieder} and
		Todorov surfaces of type $(0,9)$ or $(1,10)$. The case of K3 surfaces is
		seminal
		and is due to Beauville--Voisin \cite{BV} who establish the existence of a
		canonical 0-cycle $o$ of degree $1$ such that the modified small diagonal
		$\Gamma_3(S,o)$ vanishes. This was reinterpreted, via a direct computation, as
		saying that a K3 surface admits an MCK decomposition in \cite[\S 8]{SV}. The
		case of the Schreieder surfaces is \cite[Theorem~2]{LV}, while the case of
		Todorov surfaces is the content of Theorem~\ref{T:Todorov} to be proven below.
	\end{ex}
	
	\begin{rmk}[MCK decomposition on the image] \label{R:MCKdescend}
		Let $S$ be a smooth projective regular surface admitting a multiplicative
		Chow--K\"unneth decomposition \eqref{eqn:CKsurface}. Let $f: S\to S'$ be a
		surjective morphism to another smooth projective surface. Then $S'$ must be
		regular and it admits a multiplicative Chow--K\"unneth decomposition. Indeed,
		by
		Proposition \ref{prop:MCKsurface}, we have a degree-one 0-cycle $z$ on $S$
		such
		that $\Gamma_{3}(S, z)=0$ in $\CH^{4}(S^{3})$. It is easy to check that
		$(f,f,f)_{*}(\Gamma_{3}(S, z))=\deg(f)\Gamma_{3}(S', f_{*}(z))$. Again by
		Proposition \ref{prop:MCKsurface}, $S'$ has an MCK decomposition. Note that if
		$f: C\to C'$ is a dominant morphism of curves, the same argument shows that a
		MCK decomposition for $C$ yields an MCK decomposition for $C'$.
	\end{rmk}
	
	Proposition \ref{prop:MCKsurface} $(iii)$ and $(iv)$ give obstructions to the
	existence of multiplicative Chow--K\"unneth decompositions for regular
	surfaces.
	Moreover, we can use this to see the difference between $(iii)$ and MCK
	decompositions\,:
	\begin{prop}[MCK decomposition \emph{v.s.}\ degeneration of intersection
		product]
		\label{P:SurfaceNoMCK}
		For any $d\geq 7$, a very general smooth surface of degree $d$  in $\PP^{3}$
		does not admit
		any multiplicative Chow--K\"unneth decomposition of the form
		\eqref{eqn:CKsurface}.
		However as a very general surface has Picard number~1,
		the conditions $(iii)$ and $(iv)$ in Proposition \ref{prop:MCKsurface} are
		obviously satisfied with $z=\frac{1}{d}c_{1}(\mathcal{O}(1))^{2}$.
	\end{prop}
	
	\begin{proof}
		Generalizing the octic surface example in \cite[\S 1.4]{MR3077892}, O'Grady
		\cite{MR3522252} constructed, for each integer $d$, a smooth surface $S$ of
		degree $d$ in $\PP^{3}$ with $$\dim\operatorname{Im}\left(\CH^{1}(S)\otimes
		\CH^{1}(S)\xrightarrow{\cdot} \CH^{2}(S)\right)\geq
		\left[\frac{d-1}{3}\right].$$
		So by Proposition \ref{prop:MCKsurface} such a surface $S$ does not have an
		MCK
		decomposition of the form \eqref{eqn:CKsurface} when $d$ is at least 7. By
		Proposition \ref{prop:MCKsurface}, the modified small diagonal $\Gamma_{3}(S,
		\frac{1}{d}c_{1}(\mathcal{O}(1))^{2})$ does not vanish. However, since the
		cycle
		$\Gamma_{3}(S, \frac{1}{d}c_{1}(\mathcal{O}(1))^{2})$ is defined universally
		for
		all smooth degree-$d$ surfaces, its non-vanishing on one member, namely $S$,
		implies that it is non-trivial for a very general member by an argument using
		Hilbert schemes (\emph{cf.}\ \cite[Lemma 3.2]{Vo}).
	\end{proof}

	\subsection{On the tautological ring of powers of regular surfaces}	
	The following proposition due to Q.~Yin~\cite{Yin}, which is the analogue of
	Proposition~\ref{P:inj-curves},  gives yet another
	characterization of regular surfaces admitting an MCK decomposition.

	\begin{prop}\label{P:inj}
		Let $S$ be a smooth projective complex surface.
		Let $\operatorname{R}^\ast(S^m)\subset \CH^\ast(S^m)$ denote the subring
		generated by pullbacks of divisors, Chern classes, and diagonals $\Delta_S$.
		Assume that $S$ is regular and that $p_g := \HH^0(S,\Omega^2_S) >0$. Then $S$
		admits an MCK decomposition of the form
		\eqref{eqn:CKsurface} if and only if the
		cycle class map induces injections
		\[  \operatorname{R}^{i}(S^m)\hookrightarrow\ \HH^{2i}(S^m)\]
		for all  $m \leq 2b_{2, tr}(S)+1$  and all $i$. Here $b_{2, tr}(S)$ is the
		dimension of the transcendental cohomology of $S$, \emph{i.e.}, the smallest
		Hodge
		substructure of $\HH^2(S)$ whose complexification contains
		$\HH^0(S,\Omega^2_S)$. Moreover, $S$ admits an MCK decomposition and is
		Kimura--O'Sullivan finite-dimensional if and only if 
		$\operatorname{R}^{i}(S^m)\rightarrow\ \HH^{2i}(S^m)$ is injective
		for all ~$m$  and all~$i$.
	\end{prop}
	\begin{proof}
		This is simply an application of Yin's theorem \cite{Yin} (which works for any
		regular surface with $p_{g}>0$, see also \cite[Remark~3.8(iii)]{Yin}),
		\emph{via}
		reinterpreting the vanishing of the modified small diagonal as providing an
		MCK
		decomposition.
	\end{proof}

	\section{Cubic hypersurfaces}\label{S:cubic}

	Smooth cubic fourfolds are the most well-known examples of Fano varieties of K3
	type.  They have the following Hodge diamond\,:
	\begin{center}
		\begin{tabular}{ccccccccc}
			&&&&1&&&&\\
			&&&0&&0&&&\\
			&&0&&1&&0&&\\
			&0&&0&&0&&0&\\
			0&&1&&21&&1&&0.\\
			&0&&0&&0&&0&\\
			&&0&&1&&0&&\\
			&&&0&&0&&&\\
			&&&&1&&&&
		\end{tabular}
	\end{center}	
	Although our main aim is the case of smooth cubic hypersurfaces of dimension 4,
	it turns out that the results of this section hold for smooth cubic
	hypersurfaces of any dimension $\geq 3$.

	\subsection{Statement of the main result}
	We start with the following general definition\,:
	
	\begin{defn}[Generically defined cycle classes and the Franchetta property] 
		\label{def:Franchetta}
		Let
		$\XX\to B$ be a smooth projective morphism to a smooth
		quasi-projective complex variety $B$. 
		For any fiber $X$ of $\XX\to B$ over a closed point of~$B$, we define 
		$$\operatorname{GDCH}^*_B(X) := \operatorname{Im} (\CH^*(\mathcal X) \to
		\CH^*(X)),$$ the image of the Gysin restriction map. 
		The elements of $\operatorname{GDCH}^*_B(X)$, which by abuse we will denote
		$\operatorname{GDCH}^*(X)$ when $B$ is clear from the context, are called the
		\emph{generically defined cycles} (with respect to the deformation family $B$)
		on 
		$X$.
		
		We say that $\XX\to B$ {\em has the
			Franchetta property\/} if the restriction of the cycle class map 
		$$\operatorname{GDCH}^*_B(X) \hookrightarrow \operatorname{CH}^*(X) \to
		\HH^*(X)$$ is injective for all fibers $X$ of  $\XX\to B$ (equivalently, for a
		very general fiber $X$\,; see \cite[\S 2]{FLV}).
	\end{defn}
	
	This property is studied in \cite{OG}, \cite{PSY} for the universal family of
	K3 surfaces of low genus. This is extended to certain families of
	hyper-K\"ahler varieties in \cite{FLV} (\emph{cf.}\ also \cite{BL}), 
	and most notably for the square of the Fano variety of lines on a smooth cubic
	fourfold. 
	The aim of this section is to  generalize the latter to the Fano variey of
	lines on a smooth cubic hypersurface of \emph{any} dimension. Recall that if $X$
	is a smooth cubic hypersurface in $\PP^{n+1}$, then its Fano variety of lines,
	denoted $F$, is known to be connected smooth projective of dimension $2n-4$.

	\begin{thm}\label{flv} Let $B$
		be the open subset of $\PP\HH^0(\PP^{n+1},\mathcal{O}(3))$
		parameterizing smooth cubic hypersurfaces, and let $\FF\to B$ be the universal
		family
		of their Fano varieties of lines. Then the families 
		$\FF\to B$ and $\FF\times_B \FF\to B$
		have the Franchetta property. 	Moreover, if  $F$ is the Fano variety of lines
		on a very general cubic hypersurface $X\subseteq \PP^{n+1}$ of dimension $n> 2$,
		then
		the cycle class map induces isomorphisms
		$$\operatorname{GDCH}^*(F) \stackrel{\sim}{\longrightarrow}
		\operatorname{Hdg}(\HH^*(F)) \quad \text{and} \quad
		\operatorname{GDCH}^*(F\times F) \stackrel{\sim}{\longrightarrow}
		\operatorname{Hdg}(\HH^*(F\times F)),$$
		where $\mathrm{Hdg}(\HH^*(-))$ denotes the subalgebra generated by Hodge
		classes in $\HH^*(-)$.
	\end{thm}

	A direct consequence of Theorem~\ref{flv} is the following\,:
	
	\begin{cor} Let $Y\subset\PP^{n+1}$ be a smooth cubic hypersurface, and $F=F(Y)$ its Fano
		variety of lines. Then $F$ verifies the standard conjectures.
	\end{cor}
	
	\begin{proof} First we note that the corollary follows from the motivic
		relation in \cite{LatQuebec}\,; we provide here however a new and self-contained
		proof. The statement is clearly true if $n\leq 2$, as then $F$ has dimension
		$\leq 0$. We thus assume that $n>2$.
		From Theorem~\ref{flv}, the cycle class map induces, for the very general $Y$, 
		a surjection
		\[  \operatorname{GDCH}^*(F\times F) \ {\twoheadrightarrow}\
		\operatorname{Hdg}(\HH^*(F\times F)).\]
		Hence, for the very general $Y$ the Hodge class on $F\times F$ occuring in the
		standard conjecture of Lefschetz type for $F$ is algebraic and generically
		defined. A standard spread argument \cite[Lemma 3.2]{Vo} then allows to conclude
		the same is true for {\em any\/} smooth $Y$.
	\end{proof}

	\subsection{MCK decomposition for cubic hypersurfaces}
	Let 
	$\XX \to B$ be the universal
	family of smooth cubic hypersurfaces as above and let
	$e: \XX\to \PP^{n+1}$ be the natural evaluation map. Set
	$H:=e^{*}(c_{1}(\mathcal{O}(1))) \in \CH^1(\XX)$, the relative
	hyperplane section class. Then 
	\begin{equation}\label{eq:CKcubic}
	\pi^{2i}_\XX := \frac{1}{3} H^{n-i} \times_B H^i \quad \text{for} \	2i\neq n,
	\quad
	\text{and} \quad \pi^{n}_\XX := \Delta_{\XX/B} - \sum_{2i\neq n} \pi^{2i}_\XX
	\end{equation}
	defines a relative Chow--K\"unneth decomposition, in the sense that its
	specialization to any fiber $\XX_b$ over a closed point $b \in B$ gives a Chow--K\"unneth
	decomposition of $\XX_b$. 
	A direct consequence of Theorem~\ref{flv}, together with the fact proved in
	Proposition~\ref{prop:cubicfano} that the motive of a cubic hypersurface $X$ and
	of its symmetric square are (generically defined) direct factors of the motive
	of the Fano variety of liens $F$, is the following\,:
	
	\begin{thm}[Diaz \cite{DiazIMRN}]\label{thm:MCKcubic}
		The relative Chow--K\"unneth decomposition \eqref{eq:CKcubic} is fiberwise
		multiplicative. In particular, a smooth cubic hypersurface  admits a
		multiplicative
		Chow--K\"unneth decomposition.
	\end{thm}

	The above theorem is not stated in this form in \cite{DiazIMRN}\,; it can however be deduced from \cite[Corollary~3.3.9]{DiazIMRN}, as is explained in \S \ref{SS:proofMCK}.
In this section, we obtain Theorem~\ref{thm:MCKcubic} as a consequence of the stronger result Theorem~\ref{flv}, which will play a crucial role in our subsequent work~\cite{FLV2} where we determine, among other things,  the
Chow motive of $F$ in terms of the Chow  motive of~$X$. Moreover, contrary to the approach employed in~\cite{DiazIMRN}, our proof is independent of \cite{LatQuebec}\,; in fact in \cite{FLV2} we  strengthen,
and  obtain a new independent and more conceptual proof of, the main result
of~\cite{LatQuebec}.
The reader only interested in the case of cubic fourfolds can skip the proof of
Theorem~\ref{flv} given in \S \ref{SS:proofFranchetta} and rely
on~\cite[Theorem~1.10]{FLV}.

	\subsection{The Franchetta property for cubic hypersurfaces and their
		squares}\label{S:Franchetta-cubic}
	
	Before proceeding to the main proposition of this paragraph, we first introduce
	some notation. Given $H$ a pure Hodge structure of pure weight $w$, we denote
	$\mathrm{T}^\bullet H$ the tensor algebra on $H$, $ \mathrm{Sym}^\bullet H$ the
	symmetric algebra on $H$ with the commutativity constraint imposed by the parity
	of the weight $w$ of $H$, and $\mathrm{hdg}(H)$ for the dimension of the space
	of Hodge classes $\mathrm{Hdg}(H)$ in~$H$.
	
	\begin{lem}\label{lem:Hodgeclasses}
		Let $X\subseteq \PP^{n+1}$ be a very general cubic hypersurface of dimension
		$\neq 2$.
		Denote $H_X := \HH_{\mathrm{prim}}^n(X)(1)$ the primitive cohomology of $X$
		Tate-twisted by $1$\,; it is a pure Hodge structure of weight $n-2$. Then we
		have
		\begin{eqnarray}
		\mathrm{hdg}(H_X) =\mathrm{hdg}(H_{X}\otimes \mathrm{Sym}^{2}H_{X})& = 0,\\
		\mathrm{hdg}(H_X\otimes H_X) =  \mathrm{hdg}(\mathrm{Sym}^2H_X) &= 1,  \\
		\mathrm{hdg}(\mathrm{Sym}^2H_X \otimes \mathrm{Sym}^2H_X) &= 2.
		\end{eqnarray}
	\end{lem}
	\begin{proof}
		Given any very general hypersurface $X\subseteq \PP^{n+1}$ with cohomology
		ring not generated by Hodge classes, a theorem of Deligne
		states that the Mumford--Tate group of $H_X$ coincides with the Zariski
		closure $G$ of the
		monodromy group of a general Lefschetz pencil acting on
		$\HH^n(X)_{\mathrm{prim}}$ and is the full
		orthogonal group if $n$ is even and the full symplectic group if $n$ is odd\,;
		see \emph{e.g.}~\cite{schnell}. Since the subspace of Hodge classes in tensor
		powers of $H_X$ is given by $G$-invariants, it follows from the invariant theory
		of the orthogonal and symplectic groups that the subalgebra of Hodge classes in
		the tensor algebra $\mathrm{T}^\bullet H_X$ is generated by the quadratic form
		$q\in \mathrm{Sym}^2H_X \subset H_X^{\otimes 2}$ on $H_X$ in the case $n$ is
		even, and by the symplectic form $q\in  \mathrm{Sym}^2H_X \subset H_X^{\otimes
			2}$ on $H_X$ in the case $n$ is odd. The lemma is then straightforward.
	\end{proof}

	\begin{prop}\label{prop:Franchettacubic}
		Let $B$ be the open subset of $\PP\HH^0(\PP^{n+1},\mathcal{O}(3))$
		parameterizing smooth cubic hypersurfaces, and let $\XX\to B$ be the universal
		family. Then the families  $\XX\to B$ and $\XX\times_B \XX\to
		B$
		have the Franchetta property. Moreover, when $n\neq 2$, for a very general fiber $X$, the cycle class map induces isomorphisms
		$$\operatorname{GDCH}^*(X) \stackrel{\sim}{\longrightarrow}
		\operatorname{Hdg}(\HH^*(X)) \quad \text{and} \quad
		\operatorname{GDCH}^*(X\times X) \stackrel{\sim}{\longrightarrow}
		\operatorname{Hdg}(\HH^*(X\times X)).$$
	\end{prop}
	\begin{proof}	
		First, let us determine generators for the rings $\operatorname{GDCH}^*(X)$ and 
		$\operatorname{GDCH}^*(X\times X)$ (which were defined in Definition~\ref{def:Franchetta}). It follows from \cite[Proposition~5.2]{FLV},
		since $\XX\to B$ and $\XX\times_B \XX\to B$ are ``stratified projective
		bundles'' in the sense of \emph{loc.~cit.}, that
		$$\operatorname{GDCH}^*(X) = \langle h \rangle \quad \text{and} \quad
		\operatorname{GDCH}^*(X\times X) = \langle p_i^*h,\Delta_X \rangle,$$ 
		where $h \in \CH^1(X)$ denotes the hyperplane section class and $p_i : X\times
		X \to X$
		is the projection on the $i$-th factor.
		
		It is then clear that the cycle class map $\operatorname{GDCH}^*(X)
		\hookrightarrow \HH^*(X)$ is injective. The injectivity of the cycle  class map
		$\operatorname{GDCH}^*(X\times X) \to\HH^*(X\times X)$ is deduced
		easily from the following relation in $\CH^{n+1}(X\times X)$\,:
		\begin{equation}\label{eqn:RelationX2}
		\Delta_X\cdot p_1^*h = \Delta_X\cdot p_2^*h= \frac{1}{3}
		\sum_{\substack{i+j=n+1\\ i,j>0}}
		p_1^*h^i \cdot p_2^*h^j.
		\end{equation}
		To show \eqref{eqn:RelationX2}, we consider the following cartesian diagram
		whose excess normal bundle is $\mathcal{O}_{X}(3)$\,:
		\begin{equation*}
		\xymatrix{
			X \ar[d]\ar[r]^-{\Delta_{X}}& X\times X\ar[d]\\
			\PP^{n+1} \ar[r]& \PP^{n+1}\times \PP^{n+1}.
		}
		\end{equation*}
		The excess intersection formula \cite[Theorem 6.3]{F} yields that 
		$$\Delta_{X,*}(c_{1}(\mathcal{O}_{X}(3)))=\Delta_{\PP^{n+1}}|_{X\times X}.$$
		As $\Delta_{\PP^{n+1}}=\sum_{i=0}^{n+1}h^{i}\times h^{n+1-i}$ in
		$\CH^{n+1}(\PP^{n+1}\times \PP^{n+1})$, where with an abuse of notation $h$
		denotes also the hyperplane section class in
		$\PP^{n+1}$, we obtain that in $\CH(X\times X)$,
		$$\Delta_{X, *}(3h)=\sum_{i=0}^{n+1}h^{i}\times h^{n+1-i},$$
		which is nothing else but \eqref{eqn:RelationX2}.
		
		Suppose now that $X$ is very general of dimension $\neq 2$. That
		$\operatorname{GDCH}^*(X) \to \operatorname{Hdg}(\HH^*(X)) $ is surjective then
		follows directly from Lemma~\ref{lem:Hodgeclasses} and from the decomposition
		$$\HH^*(X) = H_X \oplus \bigoplus_{i=0}^n \QQ(-i),$$ where each summand
		$\QQ(-i)$ is spanned by $h^i$. By Lemma~\ref{lem:Hodgeclasses} again, in order
		to see that $\operatorname{GDCH}^*(X\times X) \to
		\operatorname{Hdg}(\HH^*(X\times X)) $ is surjective, it suffices to note that
		the diagonal $\Delta_X$ accounts for the Hodge class appearing in
		$\operatorname{Sym}^2 H_X$.
	\end{proof}

	\subsection{Cubic hypersurfaces and their Fano varieties of lines} 
	If $X$ is a smooth cubic hypersurface in $\PP^{n+1}$, we denote $F$ its Fano
	variety of lines.
	We first give a corrected proof of a result of Diaz
	\cite[Proposition~3.2.6]{DiazIMRN}. Let 
	$g =-c_1(\mathcal{E})|_F$ be the Pl\"ucker polarization and $c =
	c_2(\mathcal{E})|_F$,
	where	$\mathcal{E}$ is the tautological rank-2 bundle on $G:=\Gr(2,n+2)$.
	
	\begin{lem}\label{eqn:RelationInDeg=n-1}
		In $\CH^{n-1}(F)$,  
		$g^{n-1}$ is a linear combination of  $g^{n-3}c, g^{n-5}c^{2}, \dots$.
	\end{lem}
	\begin{proof}
		In the case where $n\le 4$ the result is already contained in \cite{V17}. We
		thus assume that $n\ge 5$.
		Recall that $G:=  \Gr(2,n+2)$ denotes the Grassmannian variety parameterizing projective lines
		in $\PP^{n+1}$\,; as such $\iota: F \hookrightarrow G$ is a codimension-4 closed subscheme.
		As in \cite{DiazIMRN}, for dimension reasons there exists a non-zero polynomial
		$P(x,y)\in\QQ[x,y] $ homogeneous of weighted degree $n-1$, where $x$ has degree
		$1$ and $y$ has degree $2$, such that the push-forward of $P(g,c)\in
		\CH^{n-1}(F)$ to $\CH^{n+3}(G)$ is zero, \emph{i.e.},
		\begin{equation}\label{pushf}
		P(g,c)\cdot [F]   =0\ \ \ \hbox{in}\ \CH^{n+3}(G) .
		\end{equation}
		Let us define $\R^*(F) := \iota^*\CH^*(G)\subset \CH^\ast(F)$. Since homological and numerical
		equivalence agree on $F$ (Proposition \ref{prop:stdconj} below), the push-forward map $\iota_{*}:\R^\ast(F)\to
		\CH^{\ast+4}(G)$ is injective. It follows that $P(g,c)$ gives a linear relation
		between $g^{n-1}, g^{n-3}c, \ldots$ in $\CH^{n-1}(F)$. To prove the lemma, it
		suffices to show that the polynomial $P(x,y)$ is not divisible by $y$.

		We know that $[F]=18g^2 c + 9c^2$ in $\CH^4(G)$ \cite[Example 14.7.13]{F}. We
		also know that the ideal of relations in $\CH^\ast(G)$ is generated by the following two
		polynomials
		\[ \begin{split} R_{n+1}(g,c) &=  g^{n+1} - {n\choose 1} g^{n-1}c +
		{n-1\choose 2} g^{n-3}c^2 - \cdots\ ,\\
		R_{n+2}(g,c) &=  g^{n+2} - {n+1\choose 1} g^{n}c +
		{n\choose 2} g^{n-2}c^2 - \cdots\ .\\  
		\end{split}\]
		(In the proof of  \cite[Proposition 3.2.6]{DiazIMRN} it is wrongly claimed that
		the ideal of relations in $\CH^\ast(G)$ is generated by two homogeneous
		polynomials of degree $n$ and $n+1$, respectively.)
		It thus follows from (\ref{pushf}) that there is equality of polynomials (of
		degree $n+3$)
		\begin{equation}\label{PQ} 
		P(x,y)\cdot (2 x^2 y + y^2) =  (m_1 x^2 + m_2 y) R_{n+1}(x,y) + m_3 x
		R_{n+2}(x,y) ,
		\end{equation}
		for some $m_j\in\QQ$. Let us assume, by contradiction, that the polynomial
		$P(x,y)$ is divisible by~$y$, \emph{i.e.}, that we can write
		\[ P(x,y)= p_1 x^{n-3}y + p_2 x^{n-5}y^2 + \cdots  + p_m x^{n-1-2m} y^m  ,\]
		where $p_j\in\QQ$ and $m:= \lfloor{{n-1}\over 2}\rfloor \ge 2$. Comparing the
		$x^{n+3}$ terms on both sides of (\ref{PQ}), we find that $m_3=-m_1$. Comparing
		the $x^{n+1}y$ terms on both sides of  (\ref{PQ}), we find that $m_2=-m_1$. The
		polynomial $P(x,y)$ being non-zero, we may assume that $m_1$ is non-zero. Up to
		rescaling $P(x,y)$, we may thus assume that $m_1=1$. Let us now develop equality
		(\ref{PQ}) with $m_1=1$ and $m_2=m_3=-1$. The right-hand side of (\ref{PQ}) can
		be written
		\[  \sum_{j=2}^{m+2}   a_j x^{n+3-2j} y^j  ,\]
		where the coefficient $a_j$ is
		\[ a_j= (-1)^{j+1} \Bigl(  {n+2 - j  \choose j} - {n+1-j \choose j} - {
			n+2-j\choose j-1} \Bigr) =  (-1)^{j} { n+1-j \choose j-1} .\]
		The equality of polynomials (\ref{PQ}) implies the equalities of coefficients
		\begin{equation}\label{system}
		\begin{split} 
		2p_1 &=   a_2 ,\\
		2 p_{2} + p_1& =  a_{3} ,\\
		2 p_{3} + p_2& =  a_{4} ,\\                 
		&\ \  \vdots\\
		2 p_{m} + p_{m-1}& =  a_{m+1} ,\\
		p_m &= a_{m+2} .\\
		\end{split}
		\end{equation}
		Note that the $a_j$ are integers. On the other hand, we observe
		that $p_2$ is not an integer (indeed, from (\ref{system}) we find that
		$p_1={n-1\over 2}$, and $p_2=  {1\over 2}( a_3 - p_1) =  -{1\over 4} (n^2-4 n
		+5)$, which is not an integer). It follows inductively that $p_j$ is not an
		integer for any $j\ge 2$ (indeed, from (\ref{system}) we find that $2 p_j=    
		a_{j+1} - p_{j-1}$ is not integer). But this contradicts the last line of
		(\ref{system}), and so $P(x,y)$ is not divisible by $y$.      
	\end{proof}

We now generalize  to arbitrary dimension a result of Beauville--Donagi
\cite[Proposition 6]{BD}  for cubic fourfolds.
\begin{prop}\label{prop:BDgeneral} 
	Let $X$ be a very general cubic hypersurface of dimension $n>2$. Let $F$ be its
	Fano variety of lines and $g\in \CH^{1}(F)$ be the first Chern class of the
	Pl\"ucker polarization.  Let $p: P\to F$ be the universal $\PP^{1}$-bundle and
	$q: P\to X$ be the natural evaluation morphism. Then 
	\begin{enumerate}[$(i)$]
		\item the morphism of Hodge structures $p_{*}q^{*}: \HH^{n}(X, \QQ)(1)\to
		\HH^{n-2}(F, \QQ)$ is injective, and it is an isomorphism if $n$ is odd\,;
		\item there exists a non-zero rational number $\lambda$ such that for any
		$\alpha, \alpha'\in \HH^{n}_{\prim}(X, \QQ)$, we have
		$$\lambda\langle \alpha, \alpha'\rangle_{X}=\langle p_{*}q^{*}(\alpha),
		p_{*}q^{*}(\alpha')\rangle_{F},$$
		where $\langle -, -\rangle_{X}$ is the intersection pairing on $X$ while
		$\langle \beta, \beta'\rangle_{F}:=\int_{F}\beta\cdot\beta' \cdot g^{n-2}$ for any $\beta,
		\beta'\in \HH^{n-2}(F, \QQ)$\,;
		\item The morphism $p_{*}q^{*}$ respects the primitive parts\,: 
		$p_{*}q^{*}(\HH^{n}_{\prim}(X, \QQ)(1))\subset \HH^{n-2}_{\prim}(F, \QQ)$. It is an equality if $n$ is odd or divisible by $4$\,; when $4\mid n-2$, the complement is 1-dimensional, generated by $c^{\frac{n-2}{4}}$\,;
		\item The morphism $p_{*}q^{*}$ induces an isomorphism of transcendental
		cohomology groups\,:
		$$p_{*}q^{*}: \HH^{n}_{\tr}(X, \QQ)(1)\xrightarrow{\simeq} \HH^{n-2}_{\tr}(F,
		\QQ).$$
		\item When $n$ is even, $\HH^{n-2}(F, \QQ)\cong \HH^{n}_{\prim}(X, \QQ)(1)\oplus  \QQ(-\frac{n-2}{2})^{\oplus \lfloor\frac{n+2}{4}\rfloor}$, where the isomorphism is given by $p_{*}q^{*}$ on the summand $\HH^{n}_{\prim}(X, \QQ)(1)$, and by $g^{\frac{n}{2}+1-2i}c^{i-1}$ on the $i$-th copy of $\QQ(-\frac{n-2}{2})$, for $1\leq i\leq \lfloor\frac{n+2}{4}\rfloor$.
	\end{enumerate} 
\end{prop}
\begin{proof}
	We adapt the argument of \cite{BD}.
	For $(i)$, by Shimada \cite{Shimada1}, \cite{Shimada} (see \cite[Theorem 4]{Izadi} for an alternative proof), the cylinder map $q_{*}p^{*}:
	\HH^{3n-6}(F, \QQ)\to \HH^{n}(X, \QQ)$ is surjective. Hence the transposed
	morphism $p_{*}q^{*}: \HH^{n}(X, \QQ)\to \HH^{n-2}(F, \QQ)$ is injective. When $n$ is odd, we know that $\dim\HH^{n}(X, \QQ)=\dim \HH^{n-2}(F, \QQ)$ (see for example
	\cite[Theorem 6.1]{GS}), hence $p_{*}q^{*}$  must be an isomorphism.\\
	For $(ii)$, given
	$\alpha, \alpha' \in \HH^{n}_{\prim}(X, \QQ)$, using the projective bundle formula, there
	exist $\beta_{1}, \beta_{1}'\in \HH^{n}(F, \QQ)$ and $ \beta_{2}, \beta_{2}'\in \HH^{n-2}(F, \QQ)$ such that 
	\begin{equation}\label{eqn:q*alpha}
	q^{*}(\alpha)=p^{*}(\beta_{2})\cdot \xi- p^{*}(\beta_{1}) \quad\text{ and } \quad 
	q^{*}(\alpha')=p^{*}(\beta'_{2})\cdot \xi- p^{*}(\beta'_{1}),
	\end{equation}
	where $\xi:=c_{1}(\mathcal{O}_{p}(1))=q^{*}(H)$ with $H$ denoting the hyperplane
	section class on $X$. Applying $p_{*}$ to \eqref{eqn:q*alpha}, we find that
	$$p_{*}q^{*}(\alpha)=\beta_{2}.$$
	Denote $\mathcal{E}$ the tautological rank-2 bundle on $\Gr(\PP^{1},
	\PP^{n+1})$ and  set $g:=-c_1(\mathcal{E})|_F$ to be the Pl\"ucker polarization and $c
	= c_2(\mathcal{E})|_F$.
	Multiplying \eqref{eqn:q*alpha} by $\xi=q^{*}(H)$ and using the equality
	$\xi^{2}=p^{*}(g)\xi-p^{*}(c)$,  we obtain 
	$$0=q^{*}(\alpha\cdot H)=p^{*}(\beta_{2})\cdot \xi^{2}-p^{*}(\beta_{1})\cdot
	\xi=p^{*}(\beta_{2}\cdot g)\cdot \xi-p^{*}(\beta_{2}\cdot
	c)-p^{*}(\beta_{1})\cdot \xi.$$
	Therefore
	$$\beta_{1}=\beta_{2}\cdot g \text{ and } \beta_{2}\cdot c=0,$$
	and similarly for $\beta_1'$ and $\beta_2'$.
	The identities \eqref{eqn:q*alpha} thus become
	\begin{equation}\label{eqn:q*alpha2}
	q^{*}(\alpha)=p^{*}(\beta_{2})\cdot \xi- p^{*}(\beta_{2}\cdot g ) \quad \text{ and } \quad 
	q^{*}(\alpha')=p^{*}(\beta'_{2})\cdot \xi- p^{*}(\beta'_{2}\cdot g ) .
	\end{equation}
	Taking the product of the equations in \eqref{eqn:q*alpha2}, using again 
	$\xi^{2}=p^{*}(g)\xi-p^{*}(c)$ and $\beta_{2}\cdot c=\beta_2'\cdot c=0$, one obtains
	$$q^{*}(\alpha\cdot \alpha')=p^{*}(\beta_{2}\cdot \beta_2'\cdot g^{2})-p^{*}(\beta_{2}\cdot \beta_2'\cdot g)\cdot \xi.$$
	Multiplying with $p^{*}(g^{n-3})$ and then applying $q_{*}$, we find
	$$\alpha\cdot \alpha' \cdot q_{*}p^{*}(g^{n-3})=- q_{*}p^{*}(\beta_{2}\cdot \beta_2'\cdot
	g^{n-2}).$$
	Since $q_{*}p^{*}(g^{n-3})$ is a non-zero
	multiple of the fundamental class of $X$, taking the degree of both sides  yields that 
	$$\lambda\langle\alpha, \alpha'\rangle_X=\deg(\beta_{2}\cdot \beta_2'\cdot g^{n-2}),$$ for some
	$\lambda\neq 0$. This is nothing else but the desired formula in $(ii)$.\\
	For $(iii)$, let us first show that $\beta_{2}=p_{*}q^{*}(\alpha)$ is in the primitive
	part, \emph{i.e.},~$\beta_{2}\cdot g^{n-1}=0$. However, $g^{n-1}$ is a linear
	combination of $g^{n-3}c, g^{n-5}c^{2}, \dots$ (see for example Lemma
	\ref{eqn:RelationInDeg=n-1}), hence the desired vanishing follows from the
	vanishing $\beta_{2}\cdot c=0$. The inclusion is proved. When $n$ is odd or divisible by 4, by \cite[Theorem 6.1]{GS}, $\dim \HH^{n}_{\prim}(X, \QQ)(1))=\dim \HH^{n-2}_{\prim}(F, \QQ)$, hence the inclusion must be an equality. When $4 \mid n-2$, again by \cite[Theorem 6.1]{GS}, the complement is 1-dimensional. It suffices to see that $c^{\frac{n-2}{4}}$ is not in $p_{*}q^{*}(\HH^{n}_{\prim}(X, \QQ)(1))$. Indeed, one checks easily that $c^{\frac{n-2}{4}}$ is orthogonal to $p_{*}q^{*}(\HH^{n}_{\prim}(X, \QQ)(1))$ with respect to  $\langle-, -\rangle_{F}$.\\
	For $(iv)$, it is clear that $p_{*}q^{*}$ induces a morphism between
	$\HH^{n}_{\tr}(X, \QQ)$ and $ \HH^{n-2}_{\tr}(F, \QQ)$, which is injective by
	$(i)$. To see its surjectivity, it suffices to observe that $\HH^{n}_{\tr}(X,
	\QQ)$ and $ \HH^{n-2}_{\tr}(F, \QQ)$ have the same dimension (see for example
	\cite[Theorem 6.1]{GS}).\\
	For $(v)$, by the dimension count in \cite[Theorem 6.1]{GS}, it suffices to show that $\HH^{n-2}(F, \QQ)$ is spanned by $p_{*}q^{*}(\HH^{n}_{\prim}(X, \QQ)(1))$ and $g^{\frac{n}{2}+1-2i}c^{i-1}$, for $1\leq i\leq \lfloor\frac{n+2}{4}\rfloor$, but this follows from the description of $\HH^{n-2}_{\prim}(F, \QQ)$ in $(iii)$ and the fact that $\HH^{n-4}(F, \QQ)$ is of Tate type and is generated by $g^{\frac{n}{2}-2i}c^{i-1}$, for $1\leq i\leq \lfloor\frac{n-2}{4}\rfloor$.
\end{proof}

	\begin{cor}\label{cor:composition}
		The composition of the following chain of isomorphisms is a non-zero multiple of
		the identity map\,:
		$$\HH^{n}_{\prim}(X, \QQ)\xrightarrow{p_{*}q^{*}} \HH^{n-2}_{\prim}(F,
		\QQ)\xrightarrow{\cdot g^{n-2}} \HH^{3n-6}_{\prim}(F, \QQ)\xrightarrow{q_{*}p^{*}}
		\HH^{n}_{\prim}(X, \QQ).$$
	\end{cor}
	\begin{proof}
		This is actually a general fact in linear algebra. Let $\Lambda_1$, $\Lambda_2$
		be two $\QQ$-quadratic spaces. If there is a linear isomorphism $\phi:
		\Lambda_1\xrightarrow{\cong}\Lambda_2$ and a non-zero rational number $\lambda$,
		such that $\lambda\langle x, y\rangle_{\Lambda_{1}}=\langle \phi(x),
		\phi(y)\rangle_{\Lambda_{2}}$, then the composition of the following chain of
		isomorphisms is $\lambda \cdot \id$.
		\begin{equation}
		\xymatrix{
			\Lambda_{1} \ar[r]^{\phi}&\Lambda_{2} \ar[r]^{\cong}& \Lambda_{2}^{\vee}
			\ar[r]^{{}^{t}\phi}& \Lambda_{1}^{\vee}\ar[r]^{\cong}&\Lambda_{1}\\
			x \ar@{|->}[r]& \phi(x) \ar@{|->}[r] & \langle -, \phi(x)\rangle \ar@{|->}[r]&
			\langle \phi(-), \phi(x)\rangle=\langle -, \lambda x\rangle \ar@{|->}[r]& 
			\lambda x
		}
		\end{equation}
		Thanks to Proposition \ref{prop:BDgeneral}, we can apply this general fact to
		the case when $\Lambda_{1}:=\HH^{n}_{\prim}(X, \QQ)$,
		$\Lambda_{2}:=\HH^{n-2}_{\prim}(F, \QQ)$, and $\phi:=p_{*}q^{*}$.
	\end{proof}

	We denote $\FF \to B$ the relative Fano variety of lines for the universal
	family $\XX \to B$ of cubic hypersurfaces.
	The relative Chow motives of $\XX$ and $\FF$ over $B$ can be related thanks to
	the following construction, performed
	for smooth cubic hypersurfaces of any dimension, due to Galkin--Shinder
	\cite{GS} and Voisin \cite{V15}. Let $X$ be a smooth cubic hypersurface and let
	$F$
	be its Fano variety of lines.  Galkin--Shinder \cite[Proof of Theorem
	5.1]{GS} constructed a 
	birational map
	$\phi:  X^{[2]} \dashrightarrow P_{X}$ from the Hilbert scheme of length-2
	subschemes $X^{[2]}$ of $X$,
	where $P_{X}:=\PP(T_{\PP^{n+1}}|_{X})$ is some $\PP^n$-bundle over $X$. Moreover, Voisin \cite[Proposition
	2.9]{V15} constructed an explicit resolution of indeterminacies
	\begin{equation}
  \label{diag:cubic}
  \begin{tikzcd}
    & & E \arrow[d, hook] \arrow[ddll, swap] \arrow[ddrr] \\
    & & V \arrow[dl, swap, "\phi^{1}"] \arrow[dr, "\phi^{2}"] \\
    P_2 \arrow[r, hook] \arrow[rrd, swap] & X^{[2]} \arrow[dashed, rr, "\phi"] & & P_X & P \arrow[l, hook', swap] \arrow[dll] \\
    & & F
  \end{tikzcd}
\end{equation}
	with the property  that 
	the morphism $\phi^2\colon V \to P_{X}$
	is a blow-up along a smooth center $ P\subset P_{X}$ of
	codimension $3$, where $P$ is the universal $\PP^1$-bundle
	over $F$, and the morphism
	$\phi^1\colon V\to X^{[2]}$ is a blow-up along a smooth center $
	P_{2}\subset X^{[2]}$ of codimension~$2$, where $P_{2}$ is the relative symmetric square of $P\to F$, thus has the structure of a
	$\PP^2$-bundle over $F$. Since the above construction of Galkin--Shinder and
	Voisin can
	be performed family-wise,  we obtain thanks to the projective bundle formula and
	the
	blow-up formula for Chow motives an isomorphism of
	relative motives over $B$ 
	\begin{align}\label{E:GSV}
	\h(\mathcal V) &\simeq \bigoplus_{i=0}^n \h(\XX)(-i) \oplus \h(\FF)(-3) \oplus
	\h(\FF)(-2)^{\oplus 2} \oplus \h(\FF)(-1)
	\nonumber \\
	& \simeq
	\h(\XX^{[2]}) \oplus  \h(\FF)(-3) \oplus \h(\FF)(-2) \oplus \h(\FF)(-1). 
	\end{align}
	
	From the above isomorphism, we can derive the following proposition that will be
	used in our new proof of Theorem~\ref{thm:MCKcubic}. It will also be used in the
	proof of Theorem~\ref{flv} to establish the Franchetta property of the family
	$\FF\to B$.

	\begin{prop}\label{prop:cubicfano} Let $\XX\to B$ denote the universal smooth
		cubic hypersurface in $\PP^{n+1}$ with $n>2$, $\XX^{(2)} \to B$ its relative
		symmetric square, and let
		$\FF\to B$ denote the universal Fano variety of lines in
		a cubic hypersurface. There exist relative morphisms of Chow motives over $B$
		$$\Phi : \h(\XX) \longrightarrow \h(\FF)(1) \oplus \bigoplus
		\mathds{1}_B(*),$$
		and
		$$\Psi : \h(\XX^{(2)}) \longrightarrow \bigoplus \h(\FF)(*) \oplus \bigoplus
		\mathds{1}_B(*)$$
		which are fiberwise split injective. Here, $\bigoplus \mathds{1}_B(*)$ means a
		direct sum of relative Lefschetz motives and  $\bigoplus \h(\FF)(*)$ means a
		direct sum of Tate twists of $\h(\FF)$.
	\end{prop}
	
	\begin{proof} 
		Let $\PPP\subset \FF\times_B \XX$ be the relative universal line, also seen as
		a relative morphism of motives $$\PPP : \h(\XX) \longrightarrow \h(\FF)(1).$$
		Let $\pi^j_\XX$ be the relative Chow--K\"unneth decomposition
		\eqref{eq:CKcubic} and 
		let  $\h^{n,prim}(\XX)$ denote the motive defined by the projector 
		\begin{equation*}
		\pi^{n,\prim}_\XX = 	\pi^{n}_\XX \quad \text{in case $n$ is odd}, \quad
		\text{and} \
		\pi^{n,\prim}_\XX = 	\pi^{n}_\XX  - \frac{1}{3} H^{n/2} \times_B H^{n/2} 
		\quad \text{in case $n$ is even},
		\end{equation*}
		where $H \in \CH^1(\XX)$ denotes the relative hyperplane section. 
		Corollary \ref{cor:composition} implies
		that the composition
		$$\h^{n}_{\prim}(\XX) \hookrightarrow \h(\XX) \stackrel{\PPP}{\longrightarrow}
		\h(\FF)(1) \stackrel{\cdot g^{n-2}}{\longrightarrow} \h(\FF)(-1)
		\stackrel{{}^t\PPP}{\longrightarrow} \h(X) \twoheadrightarrow
		\h^{n}_{\prim}(\XX),$$
		where $g$ denotes the relative Pl\"ucker polarization, induces fiberwise on
		$\HH^n_{\prim}(X)$ a non-zero multiple of the identity. 
		The Franchetta property 
		for $\XX\times_B \XX \to B$,  proved in
		Proposition~\ref{prop:Franchettacubic}, shows that $\pi^{n,\prim}_{\XX} \circ\,
		^t\mathcal{P} \circ (\cdot g^{n-2}) \circ \PPP\circ \pi^{n,\prim}_{\XX}$ is
		equal to
		a non-zero multiple of $\pi^{n,\prim}_{\XX}$.
		It follows that $\PPP \circ \pi^{n,\prim}_\XX : \h^{n}_{\prim}(\XX) \to
		\h(\FF)(1)$ is fiberwise split injective, from which we deduce the existence
		of
		a morphism $\Phi$ as in the statement of the proposition.
		
		Let us now turn to the existence of the morphisms $\Psi$
		as in the statement. First, recall that for any smooth projective variety
		$X$, the Hilbert scheme $X^{[2]}$ is the blow-up of the symmetric square
		$X^{(2)}$ along the diagonal. Therefore, in our relative situation, we have an
		isomorphism of relative Chow motives
		$$\h(\XX^{[2]}) \simeq \h(\XX^{(2)}) \oplus \bigoplus_{i=1}^{n-1}
		\h(\XX)(-i).$$
		Combining the above isomorphism with the isomorphism~\eqref{E:GSV} and with
		the morphism of relative motives $\Phi$ constructed above, we
		obtain the desired, fiberwise split injective, morphism $\Psi$.
	\end{proof}
	
	Let us also mention the following direct consequence of the
	isomorphism~\eqref{E:GSV}.
	\begin{prop}\label{prop:stdconj}
		Let $F$ be the Fano variety of lines on a smooth cubic hypersurface $X$. Then
		homological and numerical equivalence agree on $F$.
	\end{prop}
	\begin{proof} The statement is clearly true if $\dim X \leq 2$ since then $\dim
		F \leq 0$. We thus assume that $\dim X > 2$.
		By specializing the relative isomorphism \eqref{E:GSV} to the fiber
		corresponding to $X$, we obtain an isomorphism of motives $$M\oplus P_1 \oplus
		P_2 \simeq N \oplus P_1,$$ 
		where $M = \bigoplus_{i=0}^n \h(X)(-i), N= \h(X^{[2]}) 
		, P_1 = \h(F)(-3) \oplus \h(F)(-2) \oplus \h(F)(-1)$ and $P_2 = \h(F)(-2)$.
		Since the standard conjectures hold for smooth hypersurfaces, and since
		$X^{[2]}$ is the blow-up of $X^{(2)}$ along the diagonal, homological and
		numerical equivalence agree on $M$ and also on~$N$.  Now, comparing the
		dimension of the kernels of $\CH^*(M\oplus P_1 \oplus P_2)/\mathrm{hom} \to
		\CH^*(M\oplus P_1 \oplus P_2)/\mathrm{num}$ and of $\CH^*(N\oplus
		P_1)/\mathrm{hom} \to \CH^*(N\oplus P_1)/\mathrm{num}$, we find that
		$\ker(\CH^*( P_2)/\mathrm{hom} \to \CH^*( P_2)/\mathrm{num}) = 0$, \emph{i.e.},
		that homological and numerical equivalence agree on $F$.
	\end{proof}

	\subsection{Proof of Theorem~\ref{flv}} \label{SS:proofFranchetta}
	The case of cubic fourfolds is one of the main results of \cite{FLV}. The
	point we want to make here is that, with the input of the isomorphism of
	Galkin--Shinder--Voisin~\eqref{E:GSV}, the original proof in \cite{FLV} actually
	works in any dimension with minor adaptation. 
	Let $F$ be the Fano	variety of lines on a smooth cubic hypersurface $X$. As
	before, let us denote
	$\operatorname{GDCH}^*(F)$ and  $\operatorname{GDCH}^*(F\times F)$ the
	subrings
	of $\CH^*(F)$ and $\CH^*(F\times F)$ respectively generated by cycles that are
	restrictions of cycles on the universal families $\FF \to B$ and $\FF\times_B
	\FF \to B$.

	Our previous results \cite[Lemma~3.1 and
	Proposition~6.3]{FLV} work without change and give us generators for
	$\operatorname{GDCH}^*(F)$ and 
	$\operatorname{GDCH}^*(F\times F)$. More precisely, it is shown that these
	generically defined cycles are ``tautological''\,:
	\begin{equation}\label{eqn:GDCHgenerators}
	\operatorname{GDCH}^*(F) = \R^{*}(F):=\langle g, c \rangle \quad \text{and}
	\quad
	\operatorname{GDCH}^*(F\times F) = \R^{*}(F\times F):=\langle p_i^*g, p_j^*c,
	\Delta_F, I
	\rangle,
	\end{equation}
	where $\langle - \rangle $ means the generated subalgebra, $p_i : F\times F
	\to F$ is the projection on the $i$-th factor, $g =
	-c_1(\mathcal{E})|_F$ is the Pl\"ucker polarization, $c =
	c_2(\mathcal{E})|_F$,
	$\mathcal{E}$ is the tautological rank-2 bundle on $G: = \Gr(2,n+2)$, and $I =
	\{(\ell_1,\ell_2) \in F\times F : \ell_1\cap \ell_2 \neq \emptyset\}$ is the
	incidence correspondence.

	Let us first show the Franchetta property for $\FF\to B$. Since $\operatorname{GDCH}^*(X)$, $\operatorname{GDCH}^*(X\times X)$
	and $\operatorname{GDCH}^*(F)$ are finitely generated, we may use the
	isomorphism~\eqref{E:GSV} and proceed as in the proof of
	Proposition~\ref{prop:stdconj}: thanks to Proposition \ref{prop:Franchettacubic}, rational and
	numerical equivalence agree on $\operatorname{GDCH}^*(X)$ and on
	$\operatorname{GDCH}^*(X\times X)$, therefore the argument in Proposition~\ref{prop:stdconj} shows that rational and numerical equivalence agree
	on $\operatorname{GDCH}^*(F)$, \emph{i.e.}, that the family $\FF \to B$ has the
	Franchetta property.

	\medskip
	
	As a consequence,
	$\R^{*}(F)$ is a Gorenstein algebra\footnote{A finite dimensional graded
		algebra $A^{*}=\bigoplus_{i=0}^{d}A^{i}$ is called \emph{Gorenstein of socle
			degree $d$}, if $A^{d}$ is 1-dimensional and the multiplication $A^{i}\times
		A^{d-i}\to A^{d}$ is a perfect pairing for all $0\leq i\leq d$.} of socle degree
	$2n-4$. We can therefore bound the dimensions of the $\R^i(F)$ (we will see
	shortly that the following are actually equalities)\,:
	\begin{equation}\label{eqn:dimR(F)}
	\dim \R^{i}(F)\leq r_{i}:=\begin{cases}\lfloor \frac{i+2}{2} \rfloor, &\text{
		if } i \leq n-2\ ,\\
	\lfloor \frac{2n-2-i}{2} \rfloor, &\text{ if } i > n-2\ .
	\end{cases}
	\end{equation}
	(We note that for $i=n-1$, the above dimension estimate takes the more
	precise form of Lemma~\ref{eqn:RelationInDeg=n-1}.)
	
	In the system of generators \eqref{eqn:GDCHgenerators} for $\R^{*}(F\times F)$,
	we first observe that the generator $\Delta_{F}$ is redundant, thanks to the
	following Voisin's relation \cite{V17} which holds in any dimension with the
	same proof\,:
	\begin{equation}\label{eqn:VoisinRelation}
	I^{2}=\lambda\Delta_{F}+I\cdot\Gamma_{1}(g_{1}, g_{2})+\Gamma_{2}(g_{1},
	g_{2},c_{1},c_{2}),
	\end{equation} 
	where $\Gamma_{1}$, $\Gamma_{2}$ are polynomials and $\lambda$ is a non-zero
	rational number.
	
	Next, we have the following relations, where $\alpha, \beta$ are non-zero
	rational numbers and $P, Q, R$ are polynomials. Note that their proof,
	originally for cubic fourfolds, works without change in any dimension. 
	\begin{enumerate}[$(i)$]
		\item In \cite[Proposition 17.5]{SV}, one finds $$\Delta_{F}\cdot I=\alpha 
		c_{1}\cdot\Delta_{F}-\beta g_{1}^{2}\cdot \Delta_{F}.$$
		\item In \cite[Lemma 17.6]{SV}, there is a polynomial $P$ such that
		$$c_{1}\cdot I=P(g_{1}, g_{2},c_{1}, c_{2})\,;$$
		$$c_{2}\cdot I=P(g_{2}, g_{1},c_{2}, c_{1}).$$
		\item In \cite[Appendix]{FLV}, we proved there exists a polynomial $Q$ such
		that $$\Delta_{F,*}(g)+R(g_{1}, g_{2})\cdot I= Q(g_{1}, g_{2}, c_{1}, c_{2}).
		$$
	\end{enumerate}
	By Voisin's relation \eqref{eqn:VoisinRelation}, we can eliminate the usage of
	$\Delta_F$ in previous relations $(i)$ and $(iii)$ to obtain the following
	relations, where $P, Q, R$ are some polynomials and $\mu$ is a non-zero rational
	number.
	\begin{enumerate}
		\item[$(i')$]  $I^{3}=\mu I\cdot g_{1}^{n-2}g_{2}^{n-2}+R(g_{1}, g_{2}, c_{1},
		c_{2})$\,;
		\item[$(iii')$] $I^{2}\cdot g_{1}=I\cdot P(g_{1}, g_{2})+Q(g_{1}, g_{2}, c_{1},
		c_{2})$.
	\end{enumerate}	
	Using Lemma \ref{eqn:RelationInDeg=n-1}, and relations
	\eqref{eqn:VoisinRelation}, $(i')$, $(ii)$ and $(iii')$, we get a system of
	\emph{linear} generators for $\R^{*}(F\times F)$, namely,
	\begin{equation}\label{eqn:generators}
	\R^{*}(F\times F)=\R^{*}(F)\boxtimes \R^{*}(F)+ I\cdot \left(\Span_{\QQ}\{1,
	g, \dots, g^{n-2}\}\boxtimes\Span_{\QQ}\{1, g, \dots, g^{n-2}\}\right)+\QQ\cdot
	I^{2}.
	\end{equation}
	Combining with \eqref{eqn:dimR(F)}, we can bound the dimensions of
	$\R^{*}(F\times F)$ as follows,  where the numbers $r_{i}$ are defined in
	\eqref{eqn:dimR(F)}. We will see that these are equalities in the end.
	\begin{equation}\label{eqn:dimR(FxF)}
	\dim \R^{k}(F\times F)\leq \begin{cases}
	\sum_{i=0}^{k}r_{i}r_{k-i}, &\text{ if } 0\leq k< n-2\,;\\
	\sum_{i=0}^{k}r_{i}r_{k-i}+(k-(n-2)+1), &\text{ if } n-2\leq k< 2n-4 \,;\\
	\sum_{i=0}^{2n-4}r_{i}r_{k-i}+(n-1)+1, &\text{ if }  k= 2n-4 \,;\\
	\sum_{i=0}^{k}r_{i}r_{k-i}+(3n-6-k+1), &\text{ if } 2n-4<k\leq  3n-6 \,;\\
	\sum_{i=0}^{k}r_{i}r_{k-i}, &\text{ if } 0\leq k< n-2.
	\end{cases}
	\end{equation}
	
	To show the injectivity of the cycle class map $\R^{*}(F\times F)\to
	\HH^{2*}(F\times F)$, we bound from below the dimension of the image in each
	degree by computing the dimension of Hodge classes on~$F\times F$ for $X$ very
	general. 
	
	\begin{lem}\label{lem:Hdg}
		Let $F$ be the Fano variety of lines on a very general cubic hypersurface
		$X\subseteq \PP^{n+1}$.  Then 
		$$\mathrm{hdg}(\HH^{2i}(F))=r_{i}, \text{ for any $i$ and }$$
		\begin{equation}
		\mathrm{hdg}(\HH^{2k}(F\times F)) =\begin{cases}
		\sum_{i=0}^{k}r_{i}r_{k-i}, &\text{ if } 0\leq k< n-2\,;\\
		\sum_{i=0}^{k}r_{i}r_{k-i}+(k-(n-2)+1), &\text{ if } n-2\leq k< 2n-4 \,;\\
		\sum_{i=0}^{2n-4}r_{i}r_{k-i}+(n-1)+1, &\text{ if }  k= 2n-4 \,;\\
		\sum_{i=0}^{k}r_{i}r_{k-i}+(3n-6-k+1), &\text{ if } 2n-4<k\leq  3n-6 \,;\\
		\sum_{i=0}^{k}r_{i}r_{k-i}, &\text{ if } 0\leq k< n-2.
		\end{cases}
		\end{equation}
		Moreover, if $n\neq 2$,
		the cycle class map induces surjective morphisms
		$$\operatorname{GDCH}^*(F) {\longrightarrow}
		\operatorname{Hdg}(\HH^*(F)) \quad \text{and} \quad
		\operatorname{GDCH}^*(F\times F) {\longrightarrow}
		\operatorname{Hdg}(\HH^*(F\times F)).$$
		In particular, \eqref{eqn:dimR(F)} is an equality.
	\end{lem}
	\begin{proof}
		The Hodge structure $\HH^{*}(F, \QQ)$ was computed, for the Fano variety of
		lines on any smooth cubic hypersurface, in terms of $H_{X}:=\HH_{\operatorname{prim}}^{n}(X)(1)$ by Galkin--Shinder
		\cite[Theorem~6.1]{GS}\,:
		$$\HH^*(F) \simeq \mathrm{Sym}^2H_X \oplus \bigoplus_{i=0}^{n-2} H_X(-i) \oplus
		\bigoplus_{i=0}^{2n-4} \QQ(-i)^{a_i},
		$$
		where $a_{i} = r_i$ except for $a_{n-2} = r_{n-2}-1$. (This can be seen by
		applying $\HH^*$ to the isomorphism~\eqref{E:GSV} and using the semi-simplicity
		of the category of polarizable Hodge structures).
		The computation of $\mathrm{hdg}(\HH^{2*}(F))$ and
		$\mathrm{hdg}(\HH^{2*}(F\times F))$  is then straightforward from Lemma
		\ref{lem:Hodgeclasses}.
		
		We now show that for $X$ very general of dimension $\neq 2$, the cycle class
		map $\operatorname{GDCH}^*(F) \to \mathrm{Hdg}(\HH^*(F))$ is surjective. The
		proof is similar to that of Proposition~\ref{prop:stdconj}. With the notation
		therein, we know that the cokernels of $\operatorname{GDCH}^*(M) \to \HH^*(M)$
		and $\operatorname{GDCH}^*(N) \to \HH^*(N)$ are trivial by
		Proposition~\ref{prop:Franchettacubic}. It follows for dimension reasons that
		the cokernel of $\operatorname{GDCH}^*(P_2) \to \HH^*(P_2)$ is trivial,
		\emph{i.e.}, that the cycle class map $\operatorname{GDCH}^*(F) \to
		\mathrm{Hdg}(\HH^*(F))$ is surjective.
		
		Finally, by Lemma~\ref{lem:Hodgeclasses},  to see that the cycle class map
		$\operatorname{GDCH}^*(F\times F)\to \operatorname{Hdg}(\HH^*(F\times F))$ is
		surjective for $X$ very general of dimension $\neq 2$, it suffices to observe
		that the extra class appearing in $\operatorname{Sym}^2H_X \otimes
		\operatorname{Sym}^2H_X$ is accounted for by the diagonal $\Delta_F$.
	\end{proof}

	Now we can conclude the proof of Theorem \ref{flv}. The surjectivity of the
	restriction of the cycle class maps to generically defined cycles on $F$ or
	$F\times F$ was treated in Lemma~\ref{lem:Hdg}. We note that the Franchetta
	property for $F$ and $F\times F$ in case $\dim X \leq 2$ is trivial since then
	$F$ is either empty or zero-dimensional. Assume now that $\dim X > 2$.
	For a very general cubic hypersurface~$X$ and any integer $k$, consider the
	cycle class map 
	$$\R^{k}(F\times F)\to \HH^{2k}(F\times F).$$ 
	On the one hand, by \eqref{eqn:GDCHgenerators}, $\R^{k}(F\times F)$ consists of
	the generically defined cycles $\operatorname{GDCH}^{k}(F\times F)$, and hence
	the cycle class map is onto the space of algebraic classes in $\HH^{2k}(F\times
	F)$\,; on the other hand, by comparing \eqref{eqn:dimR(FxF)} and Lemma
	\ref{lem:Hdg}, we see that the upper bound for the dimension of $\R^{k}(F\times
	F)$ coincides with the dimension of space of algebraic classes in
	$\HH^{2k}(F\times F)$. Therefore, \eqref{eqn:dimR(FxF)} are equalities and the
	cycle class map is injective. A specialization argument then shows that the
	Franchetta property holds for $\mathcal{F}\times_{B} \mathcal{F}\to B$, thereby
	concluding the proof of Theorem~\ref{flv}.\qed

	\subsection{Proof of Theorem~\ref{thm:MCKcubic}}\label{SS:proofMCK}
	Let
	$\pi^i_{\XX}\in \CH^n(\XX\times_B \XX)$ be the relative Chow--K\"unneth
	projectors of \eqref{eq:CKcubic}. We have to show that the cycle
	\begin{equation}\label{eq:mckcubic}
	\pi^k_{\XX}\circ \delta_{\XX/B}\circ (\pi^i_{\XX}\otimes \pi^j_{\XX}) \in
	\CH^{2n}(\XX\times_B\XX\times_B\XX)
	\end{equation} vanishes fiberwise for all $k\neq i+j$. 
	Let $\sigma$ be the relative morphism $\XX\times_B \XX \to \XX\times_B \XX$
	that permutes the factors. Then, by commutativity of the algebra structure on
	the motive of a variety, we have $$\pi^k_{\XX}\circ \delta_{\XX/B}\circ
	(\pi^i_{\XX}\otimes \pi^j_{\XX})\circ \sigma = \pi^k_{\XX}\circ
	\delta_{\XX/B}\circ \sigma \circ (\pi^j_{\XX}\otimes \pi^i_{\XX})= 
	\pi^k_{\XX}\circ \delta_{\XX/B}\circ (\pi^j_{\XX}\otimes \pi^i_{\XX})$$ as
	cycles in $\CH^{2n}(\XX\times_B \XX\times_B\XX)$. It follows that the pull-push
	of
	\eqref{eq:mckcubic} along  $\XX \times_B \XX \times_B \XX \to \XX^{(2)}\times_B
	\XX$ is the cycle
	\begin{equation}\label{eq:mckcubicsym}
	\pi^k_{\XX}\circ \delta_{\XX/B}\circ (\pi^i_{\XX}\otimes \pi^j_{\XX} +
	\pi^j_{\XX}\otimes \pi^i_{\XX}) \in \CH^{2n}(\XX\times_B \XX \times_B\XX).
	\end{equation} 
	Since the cycles \eqref{eq:mckcubic} vanish cohomologically fiberwise for all
	$k\neq i+j$, the same holds for the cycles \eqref{eq:mckcubicsym} for all
	$k\neq
	i+j$. Combining Proposition~\ref{prop:cubicfano} with Theorem~\ref{flv}, we
	find
	that the cycles~\eqref{eq:mckcubicsym} vanish in the Chow group fiberwise for
	all $k\neq i+j$.
	Composing on the right with $\pi^i_{\XX}\otimes\pi^j_{\XX}$ proves the
	theorem.\qed

	\begin{rmk}[Proof of Theorem~\ref{thm:MCKcubic} using {\cite[Corollary~3.3.9]{DiazIMRN}}]
		\label{R:Diaz}
Let $X$ be a fiber of $\XX \to B$. Using \eqref{eqn:RelationX2}, one can see that the cycle 	$\pi^k_{X}\circ \delta_{X}\circ (\pi^i_{X}\otimes \pi^j_{X}) \in \operatorname{CH}^{2n}(X\times X \times X)$ belongs to the subring $\operatorname{R}^*(X\times X\times X) := \langle p_i^*h,   p_{jk}^*\Delta_X \rangle$, where $h$ is the restriction of $H$ to the fiber $X$ and where $p_i$ and $p_{jk}$ are the natural projections. It follows that the cycle $\pi^k_{X}\circ \delta_{X}\circ (\pi^i_{X}\otimes \pi^j_{X}) + \pi^k_{X}\circ \delta_{X}\circ (\pi^j_{X}\otimes \pi^i_{X})$ belongs to $\operatorname{R}^*(X^{(2)}\times X):= \operatorname{R}^*(X\times X\times X)^{\mathfrak{S}_2}$, where $\mathfrak{S}_2$ is the symmetric group acting by permuting the first two factors. Diaz shows in \cite[Corollary~3.3.9]{DiazIMRN} that the latter subring injects in cohomology. This yields that the cycle \eqref{eq:mckcubicsym} vanishes fiberwise for all $k\neq i+j$ and one concludes as before, by composing on the right with $\pi^i_{\XX}\otimes\pi^j_{\XX}$, that the cycle \eqref{eq:mckcubic} vanishes fiberwise for all $k\neq i+j$, thereby establishing Theorem~\ref{thm:MCKcubic}.
	\end{rmk}
	
	\begin{rmk}[Chern classes]\label{rmk:ChernCubic}
		By definition of $\pi^2_{\mathcal{X}}$, we have for a smooth cubic fourfold
		$X$
		that $\CH^1(X)_{(0)} = \CH^1(X) = \QQ h$. Since the Chern classes of $X$ are
		powers of $h$, we get by multiplicativity that $$c_i(X) \in \CH^i(X)_{(0)}$$
		for
		all $i$. This is especially useful in view of
		Proposition~\ref{prop:SVmachine}.
	\end{rmk}
	
	\begin{rmk}[Uniqueness of the MCK decomposition]
		Let $X$ be a smooth cubic fourfold whose Chow motive is finite-dimensional, in
		the sense of Kimura \cite{Kimura}. Then $X$ admits a unique MCK decomposition.
		Indeed, let $\{\pi_X^i,0\leq i\leq 8\}$ be a Chow--K\"unneth decomposition for
		$X$. By Kimura finite-dimensionality, since $X$ has vanishing odd cohomology,
		we
		must have $\pi_X^i=0$ for all $i$ odd. Since $\CH_0(X) = \QQ h^4$, we have by
		Remark~\ref{rmk:pi0pi2d} that $\pi_X^0 = \frac{1}{3}h^4 \times 1_{X}$. Since
		$\HH^2(X) = \QQ(-1) = \QQ h$, we have $\h^2(X) \simeq \mathds{1}(-1)$. In
		other
		words, $\pi^2_{X}$ factors as $\h(X) \to \mathds{1}(-1) \to \h(X)$. Since
		$\CH^1(X)
		=\QQ h$, we see that necessarily $\pi^2_{X} = \frac{1}{3} (h^3+\varepsilon)
		\times
		h$, where $\varepsilon$ is a homologically trivial 1-cycle on~$X$. Suppose now
		that $\{\pi_X^i,0\leq i\leq 8\}$ is an MCK decomposition. Then by
		Proposition~\ref{prop:Duality} we know it must be self-dual. In particular, we
		get $\pi^8_X = \frac{1}{3}1_{X}\times h^4$ and $\pi^6_{X} = \frac{1}{3} 
		h\times
		(h^3+\varepsilon)$. By multiplicativity, we must have $\CH^3(X)_{(0)} =
		(\CH^1(X)_{(0)})^{\cdot 3} = \QQ h^3$\,; this implies that $\varepsilon = 0$,
		and
		thereby establishes the uniqueness claim.
		
		Let now $X$ be a cubic threefold. It is known that $\CH_0(X) = \QQ h^3$ and
		consequently that $\h(X)$ is finite-dimensional in the sense of Kimura. The
		above arguments then establish unconditionally that $X$ admits a unique MCK
		decomposition. Likewise, if $X$ is a cubic fivefold, then it is known that
		$\CH_0(X) = \QQ h^5$ and $\CH_1(X) = \QQ h^4$ (see \cite{Otw}) and consequently
		\cite[Example~4.12]{Vial1} that $\h(X)$ is finite-dimensional in the sense of
		Kimura. The reader will have no trouble adapting the above arguments to show
		unconditionally that cubic threefolds and cubic fivefolds admit a unique MCK
		decomposition.
	\end{rmk}

	\section{K\"uchle fourfolds of type $c7$}\label{S:$c7$}
	As a first step towards the construction of Fano fourfolds, K\"uchle classified
	in \cite{Ku} all Fano fourfolds of index 1 that can be obtained as zero loci of
	generic sections of homogenous vector bundles on Grassmannian varieties. Among
	those K\"uchle fourfolds, those with type $c5$,  $c7$, $d3$ are of cohomological
	K3 type. 
	The aim of this section is to establish the existence of an MCK decomposition
	for  K\"uchle fourfolds of type $c7$. Those of type $d3$ were dealt with in
	\cite{d3}, while the case of those of type $c5$ is still open and certainly
	worth further study\,; see \cite{Kuz2}.
	Let us first recall the definition. 
	\begin{defn}[{\cite[Theorem 3.1]{Ku}}] A \emph{K\"uchle fourfold of type $c7$}
		is the zero locus of a generic section of the vector bundle
		\[ \wedge^2 \mathcal{Q}(1)\oplus {\mathcal O}(1) \]
		on the Grassmannian $\Gr(3,8)$, where $\mathcal{Q}$ is the tautological rank-5
		quotient bundle and $\mathcal{O}(1):=\det \mathcal Q$ is the Pl\"ucker line
		bundle.
	\end{defn}
	
	K\"uchle fourfolds of type $c7$ are of cohomological K3 type\,; their Hodge
	diamonds are as follows\,:
	\begin{center}
		\begin{tabular}{ccccccccc}
			&&&&1&&&&\\
			&&&0&&0&&&\\
			&&0&&2&&0&&\\
			&0&&0&&0&&0&\\
			0&&1&&22&&1&&0.\\
			&0&&0&&0&&0&\\
			&&0&&2&&0&&\\
			&&&0&&0&&&\\
			&&&&1&&&&
		\end{tabular}
	\end{center}
	The main result of this section is the following.	
	\begin{thm}\label{main3} A K\"uchle fourfold $X$ of type $c7$ has an MCK
		decomposition. Moreover, the Chern classes $c_j(X)$ are in
		$\CH^\ast(X)_{(0)}$.
	\end{thm}
	
	We need an alternative description, due to Kuznetsov \cite[Section 4]{Kuz1}, of
	K\"uchle fourfolds of type $c7$ as blow-ups of (special) cubic fourfolds along
	the Veronese surface. Let $M$ be the blow-up  of $\PP^5$ along the Veronese
	surface $S:=v_2(\PP^2)$, where $v_{2}: \PP^{2}\to \PP^{5}$ is the embedding
	induced by the linear system $|\mathcal{O}_{\PP^{2}}(2)|$.\footnote{In
		\cite[Section
		4]{Kuz1}, $M$ was defined to be the zero locus of a generic section of the
		vector bundle
		$\wedge^3 \mathcal{Q}$ on $\Gr(3,8)$. But according to \cite[Theorem
		4.10]{Kuz1}, this is equivalent to the definition by blow-up.} 
	Let $\pi: M\to \PP^{5}$ be the blow-up morphism.		
	It was shown in \cite[Corollary~4.11]{Kuz1} that a generic K\"uchle fourfold
	$X$ of type $c7$ arises as a
	member of the linear system $\vert 3H-E\vert$ on $M$, where $H$ is the
	pull-back of the hyperplane section class on $\PP^5$ and $E$ is the
	exceptional
	divisor. Therefore, its image $Y:=\pi(X)$ is a cubic fourfold containing $S$. 
	Set $\bar{B}:=\PP \HH^{0}(M, \mathcal{O}_{M}(3H-E))$, and set $B$ to be the
	Zariski open subset of $\bar{B}$ parameterizing smooth fourfolds $X_b\subset M$
	such that the cubic fourfold
	$Y_b=\pi(X_b)$
	is also smooth. Let $\XX\subset M\times B$ and $\YY\subset \PP^5\times B$
	denote	the universal families of K\"uchle fourfolds of type $c7$ and cubic
	fourfolds, respectively.
	
	As a first step of the proof, we establish the Franchetta property for these
	two families, which is of independent interest. Note that $\YY\to B$ is the
	universal family of Hassett's special cubic fourfolds $\mathcal{C}_{20}$.
	\begin{lem}\label{fp} The families $\XX\to B$ and $\YY\to B$ have the
		Franchetta property, in the sense of Definition~\ref{def:Franchetta}.
	\end{lem}
	
	\begin{proof} Let $\bar{\XX}\to\bar{B}$ 
		denote the family of possibly singular hypersurfaces. Since the linear
		system $|3H-E|$ is base-point free, $p\colon \bar{\XX}\to M$ has the
		structure
		of a
		projective bundle. We first show the following equality\,:
		\begin{equation}\label{univ} \ima\Bigl( \CH^\ast(\bar{\XX})\to
		\CH^\ast(X_b)\Bigr)=\ima\Bigl( \CH^\ast(M)\to \CH^\ast(X_b)\Bigr), \ \ \
		\forall b\in
		B .\end{equation}
		Indeed,  for any $\alpha\in\CH^i(\bar{\XX})$, the projective bundle
		formula yields
		\[  \alpha= \sum_{j=0}^i  p^\ast(a_{i-j})\cdot \xi^j \quad \hbox{in}\
		\CH^i(\bar{\XX}) ,\]
		where $\xi\in\CH^1(\bar\XX)$ is relatively ample with respect to $p$ and
		$a_j\in\CH^{j}(M)$.
		Let $h\in \CH^1(\bar{B})$ be a hyperplane section and let $q\colon \bar{\XX}
		\to \bar{B}$ denote the projection. We have
		$q^\ast(h)=\nu\, \xi+p^\ast(z)$, for some $\nu\in\QQ$ and $z\in \CH^1(M)$. It
		is readily checked that $\nu$ is non-zero. (Indeed, assume for a moment $\nu$
		were zero. Then we would have $q^\ast(h^{\dim\bar{B}})=p^\ast(z^{\dim
			\bar{B}})$
		in $\CH^{\dim\bar{B}}(\bar{\XX})$. But the right-hand side is zero, since
		$\dim\bar{B}>\dim M=5$, while the left-hand side is non-zero.)	
		
		The constant $\nu$ being non-zero, we can write
		\[ \xi= p^\ast(z)+q^\ast(c)\quad \hbox{in}\ \CH^1(\bar{\XX}) ,\]
		where $z\in \CH^1(M)$ and $c\in \CH^1(\bar{B})$ are non-zero elements.
		The restriction of $q^\ast(c)$ to a fiber $X_b$ is zero,
		and so we find that
		\[  {\alpha}\vert_{X_b}=  a_i^\prime\vert_{X_b}\quad \hbox{in}\
		\CH^i(X_b) ,\]
		for some $a_i^\prime\in \CH^i(M)$. This proves the equality
		(\ref{univ}).
		
		Second, we claim that for any $b\in B$, we have the following equality\,:
		\begin{equation}\label{GDCHY}
		\ima\bigl(\CH^{*}(\YY\to Y_{b})\bigr)=\ima\bigl(\CH^{*}(\PP^{5})\to
		\CH^{*}(Y_{b})\bigr)+\ima\bigl(\tau_{b,*}: \CH^{*}(S)\to 
		\CH^{*}(Y_{b})\bigr) ,
		\end{equation}
		where $\tau_{b}: S\to Y_{b}$ is the natural inclusion and the other morphisms
		are natural restriction maps.
		The right-hand side of \eqref{GDCHY} is clearly contained in its left-hand
		side. Let us show the reverse inclusion. Denote by $\pi_b\colon X_b\to Y_b$ the
		restriction of $\pi$ to $X_{b}$, which fits into the following cartesian
		commutative diagram\,:
		\begin{equation}
		\xymatrix{
			E\cup X_{b} \ar@{^{(}->}[r]^-{i_{E}\cup i_{X}}\ar[d]_{(\tau_{b}\circ\pi|_{E})
				\cup \pi_{b}}& M\ar[d]^{\pi}\\
			Y_{b} \ar@{^{(}->}[r]^{i_{Y}}& \PP^{5}.
		}
		\end{equation}
		For any $z\in \ima(\CH^{*}(\YY\to Y_{b}))$,
		$z=\pi_{b,*}(\pi_{b}^{*}(z))=\pi_{b, *}(i_{X}^{*}(\alpha))$ for some $\alpha\in
		\CH^{*}(M)$. Here we used \eqref{univ} and the fact that  $\pi_{b}^{*}(z)\in
		\CH^{*}(X_{b})$ is the restriction of a cycle of $\XX$. By the base-change
		formula, we obtain
		$$i_{Y}^{*}\pi_{*}(\alpha)=\pi_{b, *}(i_{X}^{*}(\alpha))+\tau_{b,*}
		(\pi|_{E})_{*}(i_{E}^{*}(\alpha)),$$
		where the first term on the right-hand side is $z$. The equality
		\eqref{GDCHY} is proven.
		
		Denote as before $\operatorname{GDCH}^{*}(Y_{b}):=\ima(\CH^{*}(\YY\to Y_{b}))$,
		the ``generically defined cycles''. Thanks to \eqref{GDCHY},
		$\operatorname{GDCH}^*(Y_b)$ is generated, as a $\Q$-subalgebra of
		$\CH^{*}(Y_{b})$, by the hyperplane class $h$, the fundamental class of the
		Veronese surface $S$, and the class of a line $l$ in the surface $S$. We need to
		show that the cycle class map restricted to this subalgebra
		$\operatorname{GDCH}^{*}(Y_{b})=\langle h, S, l\rangle$ is injective. 
		
		To this end, we first observe that the generator $l$ is redundant\,:
		$$l:=\tau_{*}(c_{1}(\mathcal{O}_{S}(1)))=\tau_{*}(\frac{1}{2}\tau^{*}(h))=\frac{1}{2}h\cdot
		S.$$ Hence $\operatorname{GDCH}^{*}(Y_{b})=\langle h, S\rangle$. The cohomology
		classes $[S]$ and $[h]^{2}$ being independent, we only need to establish the
		following two relations in $\CH^{*}(Y_{b})$\,:
		\begin{itemize}
			\item $S^{2}$ is proportional to $h^{4}$\,;
			\item $S\cdot h$ is proportional to $h^{3}$.
		\end{itemize}
		The first one follows immediately from the fact that $\CH_{0}(Y_{b})\simeq
		\mathbb{Q}$. For the second one, using the fact that the class of the Veronese
		surface in $\PP^{5}$ is $4c_{1}(\mathcal{O}_{\PP^{5}}(1))^{3}$, we have $$S\cdot
		h=\frac{1}{3}i_{Y}^{*}(i_{Y, *}(S))=\frac{4}{3}h^{3},$$
		where $i_{Y}: Y_{b}\to \PP^{5}$ is the natural inclusion.

		The Franchetta property for $\XX\to B$ now follows from the commutative
		diagram
		$$\xymatrix{	\CH^\ast(\XX)\ar[r] \ar[d]_{{\scriptstyle \pi_\ast}} &
			\CH^\ast(X_b) \ar[d]_ {\scriptstyle (\pi_b)_\ast}& \hspace{-25pt }\cong
			\CH^\ast(Y_b)\oplus \CH^{\ast-1}(S)\\
			\CH^\ast(\YY) \ar[r] & \CH^\ast(Y_b)
		}$$
		plus the fact that the Veronese surface $S\cong\PP^2\subset Y_b$ has
		trivial
		Chow groups.
	\end{proof}

	\begin{proof}[Proof of Theorem \ref{main3}] 
		To prove that K\"uchle fourfolds of type $c7$ have an MCK decomposition, it
		suffices, by specialization, to show that the generic K\"uchle fourfold of
		type $c7$ has an MCK
		decomposition with universally defined projectors. The generic fourfold of
		type $c7$ is a
		blow-up of a cubic fourfold as above, and so to establish an MCK, it suffices
		to check that the
		hypotheses of the ``blow-up'' part of
		Proposition \ref{prop:SVmachine} are met with. The second, third and fourth
		hypotheses
		follow from Theorem \ref{thm:MCKcubic}, Remark \ref{rmk:ChernCubic}, and the
		fact that
		$S\cong\PP^2$. The only thing that needs checking is the first hypothesis,
		\emph{i.e.},
		that the graph of the inclusion morphism $\tau\colon S\to Y$ lies in
		$\CH^\ast(S\times Y)_{(0)}$.

		Clearly, 
		the inclusion
		morphism is universally defined (\emph{i.e.}, it exists as a relative morphism
		$\tau$
		over the base $B$). Since
		\[ \Gamma_{\tau_b}\ \ \in \CH^4(S\times Y_b)\cong \bigoplus_{j=2}^4
		\CH^j(Y_b)
		\]
		and since the MCK decomposition for $S\times Y_b$ is universally defined
		over the
		base $B$, it follows from Lemma \ref{fp} that
		\[  (\pi^j_{S\times Y_b})_\ast (\Gamma_{\tau_b})=0\quad \hbox{in}\
		\CH^4(S\times Y_b)\ \ \ \forall j\not= 8,\]
		\emph{i.e.}, $\Gamma_{\tau_b}\in \CH^4(S\times Y_b)_{(0)}$, as desired.
		
		Finally, the statement about the Chern classes $c_j(X)$ follows from Lemma
		\ref{fp}, in view of the fact that the Chern classes of $X$ and the MCK
		decomposition of $X$ are universally defined.
	\end{proof}

	\begin{rmk}\label{r1}The link
		between K\"uchle $c7$ and cubic fourfolds is interesting also on the cubic
		fourfolds side. Indeed, thanks to the relation with K\"uchle $c7$ we were able
		to show (Lemma \ref{fp}) the Franchetta property for the Hassett divisor
		$\CC_{20}$ (parametrizing cubic fourfolds containing a Veronese surface).	
		More generally, one can ask whether all Hassett divisors $\CC_d$ have the
		Franchetta property. For those values of $d$ where there is an associated K3
		surface, this is equivalent to O'Grady's generalized Franchetta conjecture for
		K3 surfaces. For other values of $d$ (such as $d=20$), the expected answer is
		not so clear.
	\end{rmk}

	\section{Todorov surfaces}\label{S:todorov}
	
	We exhibit examples of regular surfaces with ample canonical class and with
	cohomology of K3 type that admit a multiplicative Chow--K\"unneth
	decomposition,
	namely Todorov surfaces of type $(0,9)$ and $(1,10)$. For that purpose we first
	state a general criterion for lifting MCK decompositions along dominant
	morphisms of regular surfaces. (Criteria for descending an MCK decomposition
	along a generically finite
	morphism were given  and used in \cite{SV, SV2} to show that MCK decompositions
	were stable under the operation of taking the Hilbert scheme of length-$2$ or
	$3$ subschemes.) We then define Todorov surfaces and check that
	for those of type $(0,9)$ and $(1,10)$ all the hypotheses of the general
	criterion are met with. As a by-product we obtain a Franchetta-type result for
	those Todorov surfaces.
	
	\subsection{A general criterion for lifting MCK decompositions on regular
		surfaces} Let $p: S\to T$ be a dominant morphism between regular surfaces. In
	Remark~\ref{R:MCKdescend}, we saw that an MCK decomposition on $S$ descends to
	give an MCK decomposition on $T$. In general, one cannot infer an MCK
	decomposition on $S$ from an MCK decomposition on $T$. (Consider indeed the
	blow-up of a K3 surface at a very general point). The following proposition
	provides sufficient conditions for an MCK decomposition on $T$ to lift to an
	MCK
	decomposition on $S$ and will be used to construct such a decomposition on
	certain Todorov surfaces.
	
	\begin{prop}[Lifting MCK decompositions]\label{P:criterionMCK}
		Let $\TT \to B$ and $\Ss \to B$ be relatively flat and projective surfaces
		over
		a smooth quasi-projective complex variety $B$. For any $b\in B$, denote by
		$T_{b}$ and $S_{b}$ the respective fibers. Assume that for all $b\in B$,
		$S_{b}$
		and $T_{b}$ are regular surfaces which are smooth or finite-group quotients of
		smooth surfaces. Let $p:
		\Ss \to \TT$ be a dominant morphism over $B$. Assume the following
		conditions\,:
		\begin{enumerate}[(i)]
			\item (Franchetta property for $\TT/B$)  There exists a cycle $c\in
			\CH^2(\TT)$
			which is  fiberwise of degree~1 such that for any $\alpha \in \CH^2(\TT)$ and
			for
			any $b\in B$, we have $\alpha_b \in \Q c_b$\,;  
			\item The relative CK decomposition $\pi^0_\TT = c\times_B \TT$, $\pi^4_\TT =
			\TT \times_B c$, $\pi^2_\TT = \Delta_{\TT/B} - \pi^0_\TT  - \pi^4_\TT $ is
			fiberwise an MCK decomposition\,; in other words, the modified small diagonal
			vanishes $\Gamma_{3}(T_{b}, c_{b})=0$ for any $b\in B$ (see Proposition
			\ref{prop:MCKsurface}).
			\item For any $b\in B$, the pushforward $p_* : \CH^{2}(S_{b}) \to
			\CH^{2}(T_{b})$ is an isomorphism. 
		\end{enumerate}
		Then, denoting $c' = \frac{1}{\deg(p)} p^*c$, the set   $\{\pi^0_\Ss =
		c'\times_B \Ss, \pi^4_\Ss =  \Ss \times_B c', \pi^2_\Ss= \Delta_{\Ss/B} -
		\pi^0_\Ss  - \pi^4_\Ss\}  $ defines a relative CK decomposition for $\Ss$
		which
		is fiberwise an MCK decomposition. Moreover, the Franchetta property holds for
		the family $\Ss \to B$.
	\end{prop}
	
	\begin{proof}
		Intersection theory with rational coefficients extends naturally to quotients
		of smooth varieties by finite groups.
		Denote by $d:=\deg(p)$.
		First of all, note that the inverse of the isomorphism in $(iii)$ is obviously
		$\frac{1}{d}p^{*}$ by the projection formula. Then we can establish the
		Franchetta property for the family $\Ss/B$\,: for any $\alpha\in \CH^{2}(\Ss)$
		and any $b\in B$, by $(i)$, $p_{*}(\alpha_{b})=mc_{b}$ in $\CH_{0}(T_{b})$ for
		some $m\in \Q$. Therefore,

		$$\alpha_{b}=\frac{1}{d}p^{*}(p_{*}(\alpha_{b}))=\frac{1}{d}p^{*}(mc_{b})=mc_{b}'\in
		\Q c_{b}'.$$
		To show the multiplicativity of the Chow--K\"unneth decomposition for the
		fibers of $\Ss$, we need to reinterpret it as follows. The relative Chow
		motive
		of $\Ss/B$ splits as 
		$$\h(\Ss) = \h(\Ss)^{\mathrm{inv}} \oplus \h(\Ss)^{\mathrm{coinv}},$$
		where $\h(\Ss)^{\mathrm{inv}} = \frac{1}{d}\, p^*p_* \h(\Ss)$ and
		$\h(\Ss)^{\mathrm{coinv}} = 
		\big( \Delta_{\Ss/B} -  \frac{1}{d}\, p^*p_*\big) \h(\Ss)$. 
		For any $b\in B$,  $\h(S_{b})^{\mathrm{inv}}$ forms a subalgebra object of
		$\h(S_{b})$, because $\frac{1}{d}\, p^*p_* \circ \delta_{S_{b}} \circ (
		\frac{1}{d}\, p^*p_* \otimes  \frac{1}{d}\, p^*p_*) = \delta_{S_{b}} \circ (
		\frac{1}{d}\, p^*p_* \otimes  \frac{1}{d}\, p^*p_*)$. Moreover $p^{*}$ induces
		an isomorphism of algebra objects\,: 
		\begin{equation}\label{eqn:T=Sinv}
		\h(T_{b})\simeq \h(S_{b})^{\mathrm{inv}}.
		\end{equation} 
		In particular, $p_{*}: \CH^{2}(\h(S_{b})^{\mathrm{inv}})\to
		\CH^{2}(\h(T_{b}))$
		is an isomorphism. The assumption $(iii)$ implies that for any $b\in B$,
		$\CH^{2}(\h(S_{b})^{\mathrm{coinv}})=0$. It is also clear that
		$\CH^{0}(\h(S_{b})^{\mathrm{coinv}})=0$. By Lemma~\ref{lemma:Tatemotive}
		below,
		$\h(S_{b})^{\mathrm{coinv}}$ is isomorphic to $\1(-1)^{r}$ for some $r\in \N$.
		Note that the number $r$ may vary with~$b$, but the minimal value is attained
		by
		a very general point $b$. 
		
		Now it is easy to see that fiberwise the relative Chow--K\"unneth
		decomposition
		for $\Ss/B$ is given as $\h^{i}(S_{b})=\h^{i}(T_{b})$ for $i=0$ and $4$, via
		the
		isomorphism \eqref{eqn:T=Sinv}, while $\h^{2}(S_{b})=\h^{2}(T_{b})\oplus
		\h(S_{b})^{\mathrm{coinv}}$. Let us show that it is multiplicative. Let $b$ be
		a very
		general point of $B$.
		
		$\bullet$ By the isomorphism \eqref{eqn:T=Sinv} and the assumption $(ii)$,
		this
		decomposition is multiplicative on the subalgebra summand
		$\h(S_{b})^{\mathrm{inv}}$.

		$\bullet$ By construction, the intersection product sends
		$\h(S_{b})^{\mathrm{coinv}}\otimes \h(S_{b})^{\mathrm{coinv}}$ to
		$\h(S_{b})^{\mathrm{inv}}$ and we need to show that this map lands in
		$\h^{4}(S_{b})$. As $\h(S_{b})^{\mathrm{coinv}}\simeq \1(-1)^{r}$, the
		morphism
		$\h(S_{b})^{\mathrm{coinv}}\otimes \h(S_{b})^{\mathrm{coinv}}\to
		\h(S_{b})^{\mathrm{inv}}$ is given by $r^{2}$ elements of
		$\CH^{2}(\h(S_{b})^{\mathrm{inv}})$. Clearly each of these 0-cycles is the
		restriction to the fiber of an element of $\CH^{2}(\Ss)$. By the Franchetta
		property established in the beginning of the proof, this element is fiberwise
		a
		multiple of $c'_{b}$. We conclude by the definition of $\h^{4}(S_{b})$.

		$\bullet$ Similarly, the intersection product sends
		$\h(S_{b})^{\mathrm{inv}}\otimes \h(S_{b})^{\mathrm{coinv}}$ to
		$\h(S_{b})^{\mathrm{coinv}}$. From the fact that $\h(S_{b})^{\mathrm{coinv}}$
		is
		a direct sum of Lefschetz motives, one sees immediately that the intersection
		product sends $\h^{0}(S_{b})\otimes \h(S_{b})^{\mathrm{coinv}}$ to
		$\h(S_{b})^{\mathrm{coinv}}$ (the fundamental class
		is a unit for the algebra structure on the motive of a variety)
		and sends $\h^{4}(S_{b})\otimes \h(S_{b})^{\mathrm{coinv}}$ to
		zero.
		It remains to
		show that the intersection product sends $\h^{2}(S_{b})^{\mathrm{inv}}\otimes
		\h(S_{b})^{\mathrm{coinv}}$ to zero. To this end, we use again that
		$\h(S_{b})^{\mathrm{coinv}}\simeq \1(-1)^{r}$ to see that the morphism
		$\h^{2}(S_{b})^{\mathrm{inv}}\otimes \h(S_{b})^{\mathrm{coinv}}\to
		\h(S_{b})^{\mathrm{coinv}}$ is given by $r^{2}$ elements of
		$\Hom(\h^{2}(S_{b})^{\mathrm{inv}}, \1)$, which is a subspace of
		$\CH^{2}(S_{b})_{\deg=0}$. Similarly as in the previous item, each of these
		0-cycles is the restriction 
		to the fiber of an element of $\CH^{2}(\Ss)$ and the Franchetta property for
		$\Ss/B$ tells us that these elements are zero.
	\end{proof}
	
	The following easy lemma, which is a special case of \cite[Cor.~2.2]{Vial} in
	the case $\pi ={}^t\pi$, 
	is used in the proof of the previous proposition. We provide a proof for the
	sake of completeness.
	\begin{lem}\label{lemma:Tatemotive}
		Let $S$ be a smooth projective regular surface and $\pi\in \CH^{2}(S\times S)$
		be a projector. If the Chow motive $M:=(S, \pi)$ has vanishing $\CH^{2}$ and
		$\CH^{0}$, then $M$ is isomorphic to a direct sum of copies of the Lefschetz
		motive $\1(-1)$.
	\end{lem}
	\begin{proof}
		The condition $\CH^{2}(M)=0$ implies that for any $x\in S$, the restriction of
		$\pi$  to the fiber $\{x\}\times S$ is zero in $\CH^{2}(S)$. By the
		Bloch--Srinivas decomposition of diagonal theorem~\cite{BS}, there exists a
		divisor $D$ of $S$, such that $\pi$ is supported in $D\times S$. As $S$ is
		assumed to be regular, $\pi$ must be of the form 
		$$\pi=\sum_{i}D_{i}\times D'_{i}+\gamma\times S,$$ for some divisors $D_{i},
		D'_{i}$ in $S$ and some 0-cycle $\gamma$ on $S$. 
		Using the fact that $\pi$ is a projector, it is of the form
		$$\pi=\sum_{i}D_{i}\times D'_{i}+\deg(\gamma)\gamma\times S.$$
		Now the hypothesis that $\CH^{0}(M)=0$ implies that $\deg(\gamma)=0$, hence
		$\pi=\sum_{i}D_{i}\times D'_{i}$, that is, $M$ is a direct sum of copies of
		$\1(-1)$.
	\end{proof}

	\subsection{Todorov surfaces}
	\label{sst}

	\begin{defn}[\cite{Mor3}, \cite{Tod2}]\label{tod}
		A {\em Todorov surface\/} is a
		canonical surface $S$ (\emph{i.e.}, a projective surface with only rational
		double
		points as singularities and with $K_S$ ample) with $p_g(S)=1$, $q=0$, and such
		that the bicanonical map $\phi_{2K_S} : S \dashrightarrow \PP^r$ fits into a
		commutative diagram 
		$$\xymatrix{S \ar[dr] \ar@{-->}[ddr]_{\phi_{2K_S}} \ar[rr]^\iota & & S \ar[dl]
			\ar@{-->}[ddl]^{\phi_{2K_S}} \\
			& S/\iota \ar@{-->}[d] & \\
			& \PP^r &
		}$$
		where $\iota\colon S\to S$ is an involution for which $S/\iota$ is birational
		to a K3 surface.
		The K3 surface obtained by resolving the singularities of $S/\iota$ will be
		called the K3 surface {\em associated to $S$\/}.
	\end{defn}
	
	\begin{defn}[\cite{Mor3}]\label{fi} The {\em fundamental invariants\/} of a
		Todorov surface $S$ are $(\alpha,k)$, where $\alpha$ is such that the
		$2$-torsion subgroup of $\hbox{Pic}(S)$ has order $2^\alpha$, and $k=K_S^2+8$.
		(The definition of $k$ is explained as the number of rational double points on a
		so-called ``distinguished partial
		desingularization'' of $S/\iota$\,; see \cite[Theorem~5.2(\rom2)]{Mor3}.)
	\end{defn}

	\begin{rmk}\label{remtod}
		By Morrison \cite[p.335]{Mor3}, there are exactly
		$11$ irreducible families of Todorov surfaces, corresponding to the
		$11$ possible values of the fundamental invariants\,:
		\[ \begin{split}(\alpha,k)\in\ \ \Bigl\{ (0,9),(0,10),&(0,11),(1,10),(1,11),\\
		&(1,12),(2,12),(2,13),(3,14),(4,15),(5,16)\Bigr\}\ .\\
		\end{split}\]
		For most of these families (in particular for types $(0,9)$ and $(1,10)$ which
		are of interest to us), examples of smooth members are given in \cite{Tod2}.
	\end{rmk}

	\subsection{The Franchetta property for Todorov surfaces of type $(0,9)$ and
		$(1,10)$}
	
	\subsubsection{Explicit descriptions}
	We now restrict attention to Todorov surfaces $S$ with fundamental invariants
	$(0,9)$ and $(1,10)$. This means that either $K_S^2=1$, or $K_S^2=2$ and the
	fundamental group of $S$ is $\ZZ/2\ZZ$ (\cite[Theorem 2.11]{CD}).
	In these two cases, there happens to be a nice explicit description of $S$ in
	terms of complete intersections in weighted projective spaces.
	
	\begin{thm}[{Catanese \cite[Theorem 3]{Cat}}, {Todorov \cite[Theorem
			6]{Tod}}]\label{ct} Let $S$
		be a Todorov surface with fundamental invariants $(0,9)$. Then
		$S_{}$ is isomorphic to a
		complete intersection of two degree-6 hypersurfaces in the weighted projective
		space
		$\PP:=\PP(1,2,2,3,3)$, invariant under the involution
		\[ \begin{split} \iota\colon\ \ \PP\ &\to\ \PP\ ,\\
		[x_0,x_1,\ldots,x_4]\ &\mapsto\  [-x_0,x_1,\ldots,x_4]\ .\\
		\end{split}\]   
		Moreover, the involution $\iota$ of $\PP$ restricts to the involution of
		Definition~\ref{tod} on $S$.
		
		Conversely, a weighted complete intersection as above with only rational
		double points as singularities is isomorphic to
		a Todorov surface with fundamental invariants $(0,9)$.
	\end{thm}

	\begin{thm}[{Catanese--Debarre \cite[Theorems 2.8, 2.9]{CD}}]\label{cd} Let $S$
		be a Todorov surface
		with fundamental invariants $(1,10)$. Then $S_{}$ is isomorphic to the
		quotient $V/\tau$, where $V$ is a complete
		intersection in  the  weighted projective space $\PP:=\PP(1, 1, 1, 2, 2)$
		having only rational double points as singularities, given by the system of
		degree-4
		equations
		\[ \begin{cases} F=z_3^2+c w^4+w^2 q(x_1,x_2)+Q(x_1,x_2)=0\ ,&\\
		G=z_4^2+c^\prime w^4+w^2 q^\prime(x_1,x_2)+ Q^\prime(x_1,x_2)=0\ .&\\
		\end{cases}  \]
		Here $[w:x_1:x_2:z_3:z_4]$ are coordinates for $\PP$, and
		$q,q^\prime$ are quadratic forms, $Q,Q^\prime$ are quartic forms without
		common factor, and $c,c^\prime$ are constants not both $0$. The involution
		$\tau\colon \PP\to \PP$ is defined as 
		\[ [w:x_1:x_2:z_3:z_4]\ \mapsto\ [-w:x_1:x_2:z_3:z_4]\ .\]
		Moreover, the involution $\iota$ of $\PP$ defined as
		\[   [w:x_1:x_2:z_3:z_4]    \ \mapsto [w:x_1:x_2:-z_3:z_4]\]
		induces the involution of Definition \ref{tod} on $S$.		
		
		Conversely, given a weighted complete intersection $V\subset\PP$ as above, the
		quotient $S:=V/\tau$ is isomorphic to a Todorov surface with
		fundamental invariants $(1,10)$.    
	\end{thm}

	\begin{rmk}\label{larger}
		As is shown in \cite{Cat} (resp. \cite{CD}), the Todorov surfaces with
		fundamental invariants $(0,9)$ (resp. $(1,10)$)  form a 12-dimensional subfamily
		of a larger 18-dimensional (resp. 16-dimensional) family of surfaces of general 
		type with $p_g=1$, $q=0$ and $K^2=1$ (resp. $K^{2}=2$ and $\pi_1=\ZZ/2\ZZ$),
		where the above explicit description using complete intersections in the
		weighted projective space is true.
		We do not know how to establish an
		MCK decomposition for these larger families of regular surfaces with $p_g=1$.
		Note that for any surface $S$ in these larger families, there is still a Hodge
		isometry $H^2_{\operatorname{tr}}(S,\QQ)\cong
		H^2_{\operatorname{tr}}(\wt{T},\QQ)$ with $\wt{T}$ a K3
		surface \cite{Mor4}, but it is not clear whether this isometry is induced by
		an algebraic
		correspondence.	
	\end{rmk}

	\subsubsection{Universal families}
	In view of Theorems \ref{ct} and \ref{cd}, let us construct universal families,
	which will play a crucial role in establishing the Franchetta property.

	\begin{notation}\label{not} 
		\noindent
		(\rom1) \textbf{(Case (0,9))}\,: Let $\PP$ be the weighted projective space
		$\PP(1,2,2,3,3)$. Let $\bar{B}\subset\PP \bigl( \HH^0(\PP, {\mathcal
			O}_{\PP}(6))^{\oplus
			2}\bigr)$
		be the linear subspace
		parameterizing pairs of (weighted) homogeneous polynomials of degree 6
		\[  F_b(x_0,\ldots,x_4) ,\quad  G_b(x_0,\ldots,x_4) ,\ \]
		where $x_0$ only occurs in even degree. Let $B$ be the Zariski open subset of
		$\bar{B}$ consisting of $b\in \bar{B}$ such that 
		\[ S_b:=\bigl\{  x\in \PP\ \vert\ F_b(x)=G_b(x)=0 \bigr\}    \]
		is a non-singular surface.   
		
		Let
		$\Ss\to B$
		denote the total space of the family over $B$ (\emph{i.e.}, by
		Theorem~\ref{ct}, the fiber $S_b$ over $b\in B$ is a smooth Todorov
		surface of type $(0,9)$), and let $\bar\Ss\to \bar B$ denote the total space
		of the family over $\bar B$. 
		We have a diagram of families 
		\[  \begin{array}[c]{ccccc}   & & \Ss & \subset & \bar{\Ss}\\
		&& \downarrow && \downarrow\\
		& & \TT &\subset& \bar{\TT}\\
		\end{array}\]
		where $\bar{\TT}$ (resp.~$\TT$) is the quotient of $\bar{\Ss}$
		(resp.~$\Ss$) under the involution $\iota$
		induced by $[x_0,x_1,\ldots,x_4]\mapsto [-x_0,x_1,\ldots,x_4]$. That is, the
		fiber $T_b$ over $b\in B$ is a ``singular K3 surface'', and its minimal
		resolution
		is the K3 surface associated to $S_b$.

		\noindent
		(\rom2) \textbf{(Case (1,10))}\,: Let $\PP$ be the weighted projective space
		$\PP(1,1,1,2,2)$ with coordinates $[w,x_1,x_2,z_3,z_4]$. Let
		$\bar B$ be the linear subspace of $\PP \bigl( \HH^0(\PP, {\mathcal
			O}_{\PP}(4))^{\oplus
			2}\bigr)$
		parameterizing pairs of weighted homogeneous equations of
		the form
		\[ \begin{cases} F_b=a z_3^2+c w^4+w^2 q(x_1,x_2)+Q(x_1,x_2) ,&\\
		G_b=a^\prime z_4^2+c^\prime w^4+w^2 q^\prime(x_1,x_2)+ Q^\prime(x_1,x_2).&\\
		\end{cases}  \]
		
		Let $B$ be the Zariski open subset of $\bar{B}$ consisting of points $b\in
		\bar{B}$ such that 
		\[ V_b:= \bigl\{  x\in \PP\ \vert\ F_b(x)=G_b(x)=0 \bigr\}    \]
		is a smooth surface and $F_b, G_b$ are as in Theorem \ref{cd} (in particular
		$aa^\prime\not=0$ and $c,c^\prime$ not both~$0$).
		
		Let
		$ \VV \to B$
		denote the total space of the family over $B$, and let
		$\Ss:=\VV/\langle\tau\rangle  \to B$
		denote the family obtained by applying the fixed point-free involution
		$\tau\times\hbox{id}_B$ to $\VV\subset \PP\times B$ (\emph{i.e.}, by
		Theorem~\ref{cd}, the fiber $S_b$ over $b\in B$ is a smooth Todorov surface of
		type $(1,10)$). Denote similarly $\bar \VV$ and $\bar \Ss:= \bar \VV/\langle
		\tau \rangle$ the total spaces of the corresponding families over $\bar B$.
		We have a diagram of families 
		\[  \begin{array}[c]{ccccc}   &  & \VV & \subset 
		&   \bar{\VV}\\
		&& \downarrow &&   \downarrow\\
		&  & \Ss & \subset &     \bar{\Ss}\\
		&& \downarrow &&   \downarrow\\
		& & \TT &\subset&    \bar{\TT}\\
		\end{array}\]
		where $\bar{\TT}$ (resp. $\TT$) is the quotient of $\bar{\Ss}$ (resp. $\Ss$)
		under the
		involution $\iota$
		induced by $[w,x_1,x_2,z_3,z_4]\mapsto [w,x_1,x_2,-z_3,z_4]$.     
		That is, the fiber $T_b$ over $b\in B$ is a ``singular K3 surface'', and its
		minimal resolution
		is the K3 surface associated to $S_b$.     
	\end{notation}

	\begin{rmk} In both cases of Notation \ref{not}, it can be checked (\emph{cf.}\
		\cite{Cat} and \cite{CD}) that the parameter space $B$ is non-empty,
		\emph{i.e.}, the general Todorov surface of type (0,9) or (1,10) is smooth. See
		also Remark~\ref{remtod}.
	\end{rmk}

	\subsubsection{The Franchetta property}
	
	\begin{prop}[Franchetta property for $T$]\label{p1} Let $\TT\to B$ be the
		universal family as above. Let $\gamma\in \CH^2(\TT)$ be a cycle that has
		degree
		$0$ on the general fiber. Then
		\[ \gamma\vert_{T_b}=0\ \ \ \hbox{in}\ \CH^2(T_b)\ \ \ \forall b\in B\ .\]
	\end{prop}
	
	\begin{proof}
		For the case $(0,9)$, the family $\TT\to B$ is constructed as the quotient
		\[  \pi\colon\ \ \Ss\ \to\ \Ss/\langle \iota \rangle=:\TT\ ,\]
		where $\Ss\to B$ is as in Notation \ref{not}$(\rom1)$, and $\iota$ is as in
		Theorem \ref{ct}. 
		The quotient 
		$ \PP/ \langle \iota \rangle $
		can be identified with the weighted projective space
		$\PP^\prime:=\PP(2,2,2,3,3)$, and quotients of the form $T_b=S_b/\langle
		\iota_b\rangle$ can be identified with weighted complete intersections of degree
		$(6,6)$ in $\PP^\prime$. It follows that $\bar{\TT}\to \bar{B}$ is the same as
		the universal family of
		weighted complete intersections of degree $(6,6)$ in $\PP^\prime$. As such,
		$\bar{\TT}\to\PP^\prime$ has the structure of a projective bundle.

		Once we have a projective bundle structure, the argument proving (\ref{univ})
		shows the following equality\,:
		\[  \ima\Bigl( \CH^\ast(\bar{\TT})\to
		\CH^\ast(T_b)\Bigr)=\ima\Bigl( \CH^\ast(\PP^\prime)\to \CH^\ast(T_b)\Bigr),
		\quad
		\forall b\in
		B.\]
		But $\CH^2(\PP^\prime)$ is one-dimensional, generated by the square of 
		a hyperplane $h\subset\PP^\prime$.
		It follows that for any $\gamma\in \CH^2(\TT)$,
		\[ \gamma\vert_{T_b}=  m\, h^2\vert_{T_b}   \ \ \ \hbox{in}\ \CH^2(T_b)\ ,\]
		where $m\in\QQ$ and $h_b:=h\vert_{T_b}\in \CH^1(T_b)$ is a hyperplane section.
		This implies the Franchetta property.
		
		For the case $(1,10)$, the family $\TT\to B$ is constructed as the quotient
		\[ \pi\colon\ \ \VV\ \to\ \VV/\langle \tau,\iota\rangle=:\TT ,\]	
		where $\VV\to B$ and $\tau$ and $\iota$ are as in $(\rom2)$ of Notation
		\ref{not}.
		We note that for any $b\in B$ the surface 
		$V_b\subset\PP$ is contained in $\PP^-:=
		\PP\setminus\{[0,0,0,1,0], [0,0,0,0,1]\}$. This means that $\VV$ is a
		Zariski open subset of $\bar{\VV}^-$, which is defined
		by the fiber diagram
		\[ \begin{array}[c]{ccc}  
		\bar{\VV}^- &\subset& \bar{\VV}\\
		\downarrow&&\downarrow\\
		\PP^-&\subset&\ \PP\ .
		\end{array}\]
		It is proven in \cite[Lemma 2.12]{tod} that $\bar{\VV}^-\to\PP^-$ is a
		$\PP^r$-bundle.       	
		The argument proving (\ref{univ}) then shows that there is equality
		\[  \ima\Bigl( \CH^\ast(\bar{\VV}^-)\to
		\CH^\ast(V_b)\Bigr)=\ima\Bigl( \CH^\ast(\PP^-)\to \CH^\ast(V_b)\Bigr)\quad
		\forall b\in
		B.\]
		But $\CH^2(\PP^-)$ is one-dimensional, generated by the square of a hyperplane
		$h$.
		It follows that
		\[   (\pi_b)^\ast(\gamma\vert_{T_b})=  ( \pi^\ast\gamma)\vert_{V_b} =
		h^2\vert_{V_b} \ \  \hbox{in}\ \CH^2(V_b)\ , \]
		and we conclude as before. 
	\end{proof}

	\begin{rmk}The link between
		Todorov surfaces and K3 surfaces is interesting also from the K3 side. Indeed,
		Proposition \ref{p1} says that
		the Franchetta property holds for the universal K3 surface obtained as a
		double cover of the projective plane branched along two cubics. (The fact that
		the quotient $T$ is of
		this type is proven in \cite{Rito}. The fact that a general such K3 is
		quotient
		of a Todorov surface of type (0,9) is because the period map for Todorov
		surfaces of type (0,9) is known to have 2-dimensional fibers \cite{Tod}.)
		This corresponds to a 10-dimensional locus inside~$\FF_2$, the moduli stack of
		K3 surfaces of degree 2, and it is not a
		priori clear that the Franchetta property should be true over this locus. This
		is similar to Remark \ref{r1} about special cubic fourfolds related to
		K\"uchle fourfolds of type $c7$.
	\end{rmk}

	\subsection{Constructing an MCK decomposition}
	We prove Theorem \ref{T:Todorov}. First, we recall the following result
	concerning the Chow group of
	0-cycles of Todorov surfaces of type $(0,9)$ and $(1,10)$. It can be seen as a
	special case of the Bloch conjecture.

	\begin{thm}[Laterveer \cite{moi}, \cite{tod}]\label{isomotives}
		Let $S$ be a smooth Todorov
		surface with fundamental invariants (0,9) or (1,10). Let $\wt{T}$ be its
		associated
		K3 surface, and let $\Gamma\in \CH^2( S\times \wt{T})$ be the correspondence
		induced by
		the quotient morphism $S\to T:=S/\iota$ and the resolution of
		singularities $\wt{T}\to T$.  Then $\Gamma$ induces an isomorphism 
		$$\Gamma_*: \CH_0(S)\ \xrightarrow{\cong}\ \CH_0(\wt{T})\ .$$ 
	\end{thm}
	
	\begin{proof}[Sketch of proof]
		The
		(0,9) case is \cite[Proposition~30]{moi}, while the
		the (1,10) case is \cite[Theorem~5.2]{tod}. 
		In both cases, the crux in proving an isomorphism of Chow groups is that there
		is
		an explicit description of the surfaces $S$ in terms of (quotients of)
		complete intersections. This ensures that Voisin's method of
		``spread''~\cite{Vo} applies. This method exploits the fact that
		the total space of the universal family of complete intersections of a certain
		type has a very simple structure.
		Thanks to this simple structure, one can prove the Franchetta property for
		self-correspondences of degree zero that exist universally. This applies in
		particular to (a modification of) the graph of the involution $\iota$.
	\end{proof}

	\begin{proof}[Proof of Theorem \ref{T:Todorov}]
		This is an application of Proposition \ref{P:criterionMCK}, with $\Ss\to B$
		the family of smooth Todorov surfaces, and $\TT\to B$ the family of
		quotients under the involution $\iota$.
		Let us ascertain that the hypotheses of Proposition~\ref{P:criterionMCK} are
		met
		with.
		The Franchetta property for $\TT\to B$ is Proposition~\ref{p1}. The MCK
		decomposition for the ``singular K3 surface'' $T$ follows from the
		MCK decomposition for its minimal resolution of singularities $\wt{T}$ (which
		is
		a K3 surface and hence admits an MCK decomposition by
		Example~\ref{E:surfaces})
		via Remark \ref{R:MCKdescend} (which still makes sense for surfaces with
		quotient singularities). Finally, hypothesis (\rom3) is Theorem
		\ref{isomotives}.
	\end{proof}

	\begin{rmk} Theorem \ref{isomotives} has recently been proven in \cite{Z} for
		Todorov surfaces
		with fundamental invariants (2,12). As such, it seems likely that the present
		approach also works to establish an MCK decomposition for this third family of
		Todorov surfaces.
	\end{rmk}

	\bibliographystyle{amsalpha}

\begin{thebibliography}{9}
		
		\bibitem{andre} Y.~Andr\'e, Une introduction aux motifs (motifs purs, motifs
		mixtes, p\'eriodes), Panoramas et Synth\`eses, 17. Soci\'et\'e Math\'ematique
		de France, Paris, 2004. xii+261 pp.
		
		\bibitem{Beau} A.~Beauville, Sur l'anneau de Chow d'une vari\'et\'e
		ab\'elienne, Math. Ann. 273 (1986), 647--651.
		
		\bibitem{Beau3} A.~Beauville, On the splitting of the Bloch--Beilinson
		filtration, in: Algebraic cycles and motives (J. Nagel et al., eds.), London
		Math. Soc. Lecture Notes 344, Cambridge University Press 2007.
		
		\bibitem{BD} A.~Beauville and R.~Donagi, La vari\'et\'e des droites d'une
		hypersurface cubique de dimension 4,
		C. R. Acad. Sci. Paris Sér. I Math. 301 (1985), no. 14, 703--706.
		
		\bibitem{BV} A.~Beauville and C.~Voisin, On the Chow ring of a K3 surface, J.
		Alg. Geom. 13 (2004), 417--426.
		
		\bibitem{BL} N.~Bergeron and Z.~Li,
		\newblock Tautological classes on moduli spaces of hyper-{K}\"{a}hler
		manifolds,
		\newblock { Duke Math. J.}, 168(7):1179--1230, 2019.
		
		\bibitem{BS} S.~Bloch and V.~Srinivas, Remarks on correspondences and
		algebraic cycles, American Journal of Mathematics Vol. 105, No 5 (1983),
		1235--1253.
		
		\bibitem{Cat} F.~Catanese, Surfaces with $K^2=p_g=1$ and their period mapping,
		in:  Algebraic geometry (Copenhagen, 1978), Springer Lecture Notes in
		Mathematics, Springer 1979.
		
		\bibitem{CD} F.~Catanese and O.~Debarre, Surfaces with $K^2=2$, $p_g=1$,
		$q=0$, J. reine u. angew. Math. 395 (1989), 1--55.
		
		\bibitem{Ceresa} G.~Ceresa, {$C$}\ is not algebraically equivalent to
		{$C^{-}$}\ in its
		{J}acobian, Ann. of Math. (2), 117(2):285--291, (1983).
		
		\bibitem{DV} O.~Debarre and C.~Voisin, Hyper-K\"ahler fourfolds and Grassmann
		geometry, J. reine angew. Math. 649 (2010), 63--87.
		
		\bibitem{DM} C.~Deninger and J.~Murre, Motivic decomposition of abelian
		schemes and the {F}ourier transform, J. Reine Angew. Math., 422:201--219,
		(1991).
		
		\bibitem{DiazIMRN} H.A.~Diaz, The Chow ring of a cubic hypersurface, I.M.R.N.
		(2019). 
		
		\bibitem{FM} E.~Fatighenti and G.~Mongardi, Fano varieties of K3 type and IHS
		manifolds, arXiv:1904.05679.
		
		\bibitem{MR3077892} L.~Fu, Decomposition of small diagonals and {C}how rings
		of hypersurfaces and {C}alabi-{Y}au complete intersections, Adv. Math., 244:894--924, (2013).
		
		\bibitem{FLV} L.~Fu, R.~Laterveer, and Ch.~Vial, with a joint appendix with
		M.~Shen, The generalized {F}ranchetta conjecture for some hyper-{K}{\"a}hler
		varieties, Journal de Math\'ematiques Pures et Appliqu\'ees  130 (2019), 1--35.
		
		\bibitem{FLV2} L.~Fu,  R.~Laterveer, and Ch.~Vial, The
		generalized {F}ranchetta conjecture for some hyper-{K}{\"a}hler
		varieties, II, Journal de l'\'Ecole polytechnique, Tome 8 (2021), 1065--1097. 
		
		\bibitem{FTV} L.~Fu, Z.~Tian and Ch.~Vial, Motivic hyper-{K}\"{a}hler
		resolution conjecture, {I}: generalized
		{K}ummer varieties, Geom.  Topol., 23(1):427--492, (2019).
	
		\bibitem{FuVialJAG} L.~Fu and Ch.~Vial, Distinguished cycles on varieties with
		motive of abelian type and the  section property, Journal of Algebraic
		Geometry 29 (2020), 53--107.
		
		\bibitem{FV} L.~Fu and Ch.~Vial, Cubic fourfolds,
		Kuznetsov components and Chow motives, arXiv:2009.13173.
		
		\bibitem{F} W.~Fulton, Intersection theory, Springer-Verlag Ergebnisse der
		Mathematik, Berlin Heidelberg New York Tokyo 1984.
		
		\bibitem{GS} S.~Galkin and E.~Shinder, The Fano variety of lines and
		rationality problem for a cubic hypersurface, arXiv:1405.5154.
	
		\bibitem{GrossSchoen} B.~Gross and C.~Schoen, The modified diagonal cycle on
		the triple product of a pointed curve, Ann. Inst. Fourier (Grenoble) 45
		(1995),	no. 3, 649--679.
		
		\bibitem{IM} A.~Iliev and L.~Manivel, Fano manifolds of degree $10$ and EPW
		sextics, Ann. Sci. Ecole Norm. Sup. 44 (2011), 393--426.
		
		\bibitem{IM2} A.~Iliev and L.~Manivel, Hyperk\"ahler manifolds from the
		Tits--Freudenthal square, European Journal of Mathematics 5 (2019),
		1139--1155.
		
		\bibitem{IKKR} A.~Iliev, G.~Kapustka, M.~Kapustka, and K.~Ranestad, EPW cubes,
		J. Reine Angew. Math. 748 (2019), 241--268. 
		
		\bibitem{Izadi} E.~Izadi, A Prym construction for the cohomology of a cubic hypersurface, Proceedings of the London Mathematical Society, Vol. 79, Issue 3, (1999), 535--568.
		
		\bibitem{J} U.~Jannsen, Motivic sheaves and filtrations on Chow groups, in:
		Motives (U.~Jannsen et alii, eds.), Proceedings of Symposia in Pure
		Mathematics Vol. 55 (1994), Part 1.
		
		\bibitem{Kimura} S.-I.~Kimura, Chow groups are finite dimensional, in some
		sense, Math. Ann., 331(1):173--201, (2005).
		
		\bibitem{Ku} O.~K\"uchle, On Fano $4$-folds of index $1$ and homogeneous
		vector bundles over Grassmannians, Math. Zeitschrift 218 (1995), 563--575.
		
		\bibitem{KM} A.~Kuznetsov and D.~Markushevich, Symplectic structures on moduli
		spaces of sheaves via the Atiyah class, J. Geom. Phys. 59 (2009), 843--860.
		
		\bibitem{Kuz1} A.~Kuznetsov, On K\"uchle varieties with Picard number greater
		than $1$, Izvestiya RAN: Ser. Mat. 79:4 (2015), 57--70 (in Russian);
		translation in Izvestiya: Mathematics 79:4 (2015), 698--709.
		
		\bibitem{Kuz2} A.~Kuznetsov, K\"uchle fivefolds of type $c5$, Math. Z. (2016)
		284, 1245--1278.
		
		\bibitem{MR1265530} K.~K{\"u}nnemann, On the {C}how motive of an abelian
		scheme, in: Motives ({S}eattle,
		{WA}, 1991), volume~55 of Proc. Sympos. Pure Math., pages 189--205. Amer.
		Math. Soc., Providence, RI, 1994.
		
		\bibitem{moi} R.~Laterveer, Some results on a conjecture of
		Voisin for surfaces of geometric genus one, Boll. Unione Mat. Italiana 9 no. 4
		(2016), 435--452.
		
		\bibitem{LatQuebec} R.~Laterveer,
		A remark on the motive of the Fano variety of lines of a cubic,
		Ann. Math. Qu\'e. 41 (2017), 141--154.
		
		\bibitem{tod} R.~Laterveer, Algebraic cycles and Todorov surfaces, Kyoto
		Journal of Math. 58 no. 3 (2018), 493--527.
		
		\bibitem{d3} R.~Laterveer, A remark on the Chow ring of K\"uchle fourfolds of
		type $d3$, Bulletin Australian Math. Soc. 100 no. 3 (2019), 410--418.
		
		\bibitem{Ver} R.~Laterveer, Algebraic cycles and Verra fourfolds, Tohoku Math.
		J. 72 no. 3 (2020), 451--485.	
		
		\bibitem{S2} R.~Laterveer, On the Chow ring of Fano varieties of type S2, Abh. Math. Semin. Univ. Hambg. 90 (2020), 17--28.
		
		
		\bibitem{B1B2} R.~Laterveer, On the Chow ring of certain Fano fourfolds,
		Annales Univ. Paed. Cracoviensis Studia Math. 19 (2020), 39--52.
			
		\bibitem{LV} R.~Laterveer and Ch.~Vial, On the Chow ring of Cynk--Hulek
		Calabi--Yau varieties and Schreieder varieties, Canadian Journal of Math. 72 no. 2 (2020), 505--536.		
		
		\bibitem{LLSS} Ch.~Lehn, M.~Lehn, Ch.~Sorger, and D.~van Straten,
		Twisted cubics on cubic fourfolds, J. Reine Angew. Math. {731}, 87--128,
		2017.
		
		\bibitem{LSV} R.~Laza, G.~Sacc\`a, and C.~Voisin, A hyper-K\"ahler
		compactification of the intermediate Jacobian fibration associated with a cubic 4-fold, Acta
		Math. 218 (2017), 55--135.
		
		
		\bibitem{MR2414143} G.~Marini, Tautological cycles on {J}acobian varieties,
		Collect. Math., 59(2):167--190, (2008).
		
		\bibitem{Mor3} D.~Morrison, On the moduli of Todorov surfaces, in: Algebraic
		Geometry and Commutative Algebra in Honor of Masayoshi Nagata (H. Hijikata et
		al., eds.), vol. 1, Kinokuniya, Tokyo 1988.
		
		\bibitem{Mor4} D.~Morrison, Isogenies between algebraic surfaces with
		geometric genus one,
		Tokyo Journal of Mathematics 10.1 (1987), 179--187.
		
		\bibitem{Mur} J.~Murre, On a conjectural filtration on the Chow groups of an
		algebraic variety, parts I and II, Indag. Math. 4 (1993), 177--201.
		
		\bibitem{NOY} A.~Negu\c{t}, G.~Oberdieck and Q.~Yin, Motivic decompositions for the Hilbert scheme of points of a K3 surface, preprint, arXiv:1912.09320.

		\bibitem{OG} K.~O'Grady, Moduli of sheaves and the Chow group of K3
		surfaces, Journal de Math. Pures et Appliqu\'ees 100 no. 5 (2013), 701--718.
		
		\bibitem{MR3522252}
		K.~O'Grady, Decomposable cycles and {N}oether-{L}efschetz loci, Doc. Math.,
		21:661--687, (2016).
		
		\bibitem{OG06} K.~O'Grady, Irreducible symplectic 4-folds and
		Eisenbud--Popescu--Walter sextics, Duke Math. J. 134 (2006), no. 1, 99--137.
		
		\bibitem{Otsubo}
		N.~Otsubo, On the {A}bel--{J}acobi maps of {F}ermat {J}acobians,
		Math. Z., 270(1-2):423--444, 2012.
		
		\bibitem{Otw} A.~Otwinowska, 
		Remarques sur les groupes de Chow des hypersurfaces de petit degr\'e,
		C. R. Acad. Sci. Paris S\'er. I Math. 329 (1999), no. 1, 51--56.
		
		\bibitem{PSY} N.~Pavic, J.~Shen and Q.~Yin, On O'Grady's generalized
		Franchetta conjecture, Int. Math. Res. Notices (2016), 1--13.
		
		\bibitem{Rie14} U.~Rie\ss, On the {C}how ring of birational irreducible
		symplectic varieties,
		Manuscripta Math.,  145(3-4):473--501, 2014.
		
		\bibitem{Rito} C.~Rito, A note on Todorov surfaces, Osaka J. Math. 46 no. 3
		(2009), 685--693.
		
		\bibitem{schnell} Ch.~Schnell, Two lectures about Mumford-Tate groups,
		Rend. Semin. Mat. Univ. Politec. Torino 69 (2011), no. 2, 199--216.
		
		\bibitem{Schreieder} S.~Schreieder, 
		On the construction problem for Hodge numbers,
		Geom. Topol. 19 (2015), no. 1, 295--342.
		
		\bibitem{SV} M.~Shen and Ch.~Vial, The Fourier transform for certain
		hyperK\"ahler fourfolds, Memoirs of the AMS 240 (2016), no.1139.
		
		\bibitem{SV2} M.~Shen and Ch.~Vial, On the motive of the Hilbert cube
		$X^{[3]}$, Forum Math. Sigma 4 (2016).
		
		\bibitem{Shimada1} I.~Shimada, 
		On the cylinder isomorphism associated to the family of lines on a hypersurface, J. Fac. Sci. Univ. Tokyo Sect. IA Math. 37 (1990), 703--719.
		
				
		\bibitem{Shimada} I.~Shimada, On the cylinder homomorphisms of
		Fano complete intersections, J. Math. Soc. Japan, Vol. 42, No. 4 (1990), 619--638.		
		
		\bibitem{Tavakol} M.~Tavakol, Tautological classes on the moduli space of
		hyperelliptic curves
		with rational tails, J. Pure Appl. Algebra 222 (2018), no. 8, 2040--2062.
		
		\bibitem{Tod} A.~Todorov, Surfaces of general type with $p_g= 1$ and $(K,K) =
		1$, Ann. Sci. de l'Ecole Normale Sup. 13 (1980), 1--21.
		
		\bibitem{Tod2} A.~Todorov, A construction of surfaces with $p_g=1$, $q=0$ and
		$2\le (K^2)\le 8$: counterexamples of the global Torelli theorem, Invent.
		Math. 63 (1981), 287--304.
		
		\bibitem{Vial1} Ch.~Vial, 
		Projectors on the intermediate algebraic Jacobians. 
		New York J. Math. 19 (2013), 793--822.
		
		\bibitem{V4} Ch.~Vial, Niveau and coniveau filtrations on cohomology groups
		and Chow groups, Proceedings of the LMS 106(2) (2013), 410--444.
		
		\bibitem{Vial} Ch.~Vial, Chow--K\"unneth decomposition for 3- and 4-folds
		fibred by varieties with trivial Chow group of zero-cycles, J. Algebraic Geom.
		24 (2015), 51--80.
		
		\bibitem{V6} Ch.~Vial, On the motive of some hyperk\"ahler varieties, J. fur
		Reine u. Angew. Math. 725 (2017), 235--247.
		
		\bibitem{V17} C.~Voisin, On the Chow ring of certain algebraic hyperk\"ahler
		manifolds, Pure Appl. Math. Q. 4 no. 3 part 2 (2008), 
		613--649.
		
		\bibitem{Vo} C.~Voisin, Chow Rings, Decomposition of the Diagonal, and the
		Topology of Families, Princeton University Press, Princeton and Oxford, 2014.
		
		\bibitem{V15} C.~Voisin, On the universal $CH_0$ group of cubic hypersurfaces,
		Journal Eur. Math. Soc. 19 no. 6 (2017), 1619--1653.
		
		\bibitem{Yin} Q.~Yin, Finite-dimensionality and cycles on powers of K3
		surfaces, Comment. Math. Helv. 90 (2015), 503--511.
		
		\bibitem{Z} N.~Zangani, Algebraic cycles on Todorov surfaces of type (2,12),
		arXiv:1905.10123.
	\end{thebibliography}

\end{document}